\documentclass[11pt,twoside]{article}
\usepackage[margin=1.3in]{geometry}

\usepackage{latexsym,amsfonts,amsmath,amsthm,amssymb,makeidx}
\usepackage[title]{appendix}
\usepackage{CJK,CJKnumb,CJKulem,times,dsfont,ifthen,mathrsfs,latexsym,amsfonts, color}
\usepackage{amsmath,amsthm,makeidx,fontenc,amssymb,bm,graphicx,psfrag,listings, curves,extarrows}
\usepackage{cite}

\let\oldbibliography\thebibliography
\renewcommand{\thebibliography}[1]{%
\oldbibliography{#1}%
\setlength{\itemsep}{0pt}%
}

\makeindex
\newtheorem{definition}{Definition}[section]
\newtheorem{theorem}{Theorem}[section]
\newtheorem{lemma}{Lemma}[section]
\newtheorem{corollary}{Corollary}[section]
\newtheorem{proposition}{Proposition}[section]
\newtheorem{remark}{Remark}[section]

\newcommand{\al}{\alpha}

\newcommand{\bt}{\begin{theorem}}
\newcommand{\et}{\end{theorem}}
\newcommand{\bl}{\begin{lemma}}
\newcommand{\el}{\end{lemma}}
\newcommand{\bd}{\begin{definition}}
\newcommand{\ed}{\end{definition}}
\newcommand{\bc}{\begin{corollary}}
\newcommand{\ec}{\end{corollary}}
\newcommand{\bp}{\begin{proof}}
\newcommand{\ep}{\end{proof}}
\newcommand{\bx}{\begin{example}}
\newcommand{\ex}{\end{example}}
\newcommand{\bi}{\begin{exercise}}
\newcommand{\ei}{\end{exercise}}
\newcommand{\bo}{\begin{prop}}
\newcommand{\eo}{\end{prop}}
\newcommand{\br}{\begin{remark}}
\newcommand{\er}{\end{remark}}
\newcommand{\be}{\begin{equation}}
\newcommand{\ee}{\end{equation}}
\newcommand{\ba}{\begin{align}}
\newcommand{\ea}{\end{align}}
\newcommand{\bn}{\begin{enumerate}}
\newcommand{\en}{\end{enumerate}}
\newcommand{\bg}{\begin{align*}}
\newcommand{\bcs}{\begin{cases}}
\newcommand{\ecs}{\end{cases}}

\newcommand{\sg}{\sigma}
\newcommand{\bean}{\begin{eqnarray*}}
\newcommand{\eean}{\end{eqnarray*}}

\numberwithin{equation}{section}

\begin{document}
\title{{\bf  Sharp blow up estimates  and precise asymptotic  behavior  of  singular   positive   solutions  to  fractional Hardy-H\'enon   equations}}  

\date{}
\author{\\{\bf Hui  Yang\thanks{Department of Mathematics, The Hong Kong University of Science and Technology, Clear Water Bay, Kowloon, Hong Kong, China.    E-mail address: mahuiyang@ust.hk} ,\;\;  Wenming Zou\thanks {Department of Mathematical Sciences, Tsinghua University, Beijing 100084, China. W. Zou was  supported by  NSFC. 
  E-mail address:  zou-wm@mail.tsinghua.edu.cn  }}\\
}

\maketitle
\begin{center}
\begin{minipage}{150mm}
\begin{center}{\bf Abstract}\end{center}
In this paper, we study the asymptotic behavior of positive  solutions of the fractional Hardy-H\'enon equation
$$
(-\Delta)^\sigma u = |x|^\alpha u^p ~~~~~~~~~~~  \textmd{in} ~ B_1 \backslash \{0\}
$$
with an isolated singularity at the origin, where $\sigma \in (0,  1)$ and the punctured unit ball $B_1 \backslash \{0\} \subset \mathbb{R}^n$ with $n \geq 2$.  When $-2\sigma < \alpha < 2\sigma$ and $\frac{n+\alpha}{n-2\sigma} < p < \frac{n+2\sigma}{n-2\sigma}$,  we give a  classification of isolated singularities of positive solutions,  and in particular,  this implies  sharp blow up estimates of singular solutions.    Further,  we  describe the precise asymptotic behavior of solutions near the singularity.   More generally,   we  classify  isolated boundary singularities and describe the precise asymptotic behavior of singular  solutions for a relevant degenerate elliptic equation with a nonlinear Neumann boundary condition.       
These results parallel those known for the Laplacian counterpart proved by  Gidas and Spruck (Comm. Pure Appl. Math. 34: 525-598, 1981),  but the methods are very different, since the ODEs analysis is a missing ingredient in the fractional case.  Our proofs are  based on a monotonicity formula,  combined with blow up (down) arguments,  Kelvin transformation and uniqueness of solutions of related degenerate equations on  $\mathbb{S}^{n}_+$.  We also investigate  isolated singularities located at infinity  of  fractional Hardy-H\'enon equations.    
 
\vskip0.10in 


\vskip0.10in


\end{minipage}

\end{center}

\vskip0.10in

\section{Introduction and Main Results}
In the classical paper \cite{G-S}, Gidas and Spruck studied the asymptotic behavior of positive  solutions of the following  equation 
\begin{equation}\label{GS01}
-\Delta  u= |x|^\alpha u^p~~~~~~~~~~~\textmd{in}~ B_1 \backslash  \{0\}
\end{equation}
with an isolated singularity at the origin, where the punctured unit ball $B_1 \backslash \{0\} \subset \mathbb{R}^n$ with $n \geq 3 $. Eq. \eqref{GS01} is usually called the Hardy ({\it resp.,} Lane-Emden, or H\'enon) equation for $\al <0$ ({\it resp.,} $\al=0$, $\al >0$). 
More specifically, assume 
$$
-2 < \alpha <2 ~~~~~~~~  \textmd{and} ~~~~~~~~ \frac{n + \alpha}{n-2} < p < \frac{n+2}{n-2}.
$$
Let $u \in C^2(B_1 \backslash  \{0\})$ be  a positive  solution of \eqref{GS01}.   Gidas-Spruck \cite{G-S}  proved that either the singularity at $x=0$ is removable, or there exist positive constants $c_1, c_2$ such that 
\begin{equation}\label{GSE-09}
\frac{c_1}{|x|^{(2+\al)/(p-1)}} \leq u(x) \leq \frac{c_2}{|x|^{(2+\al)/(p-1)}} ~~~~~~~~~  \textmd{near} ~~ x=0. 
\end{equation}
Further,  assume additionally  that $ p \neq \frac{n+2+2\alpha}{n-2}$, 
then they used the ODEs method  and the sharp blow up  estimate  \eqref{GSE-09} to derive the precise asymptotic behavior of  singular  solutions of \eqref{GS01} 
$$
|x|^{(2+\alpha)/(p-1)} u(x) \to C_0 ~~~~~~~~  \textmd{as} ~ x \to 0, 
$$
where
$$
C_0=\left\{  \frac{(2+\alpha)(n-2)}{(p-1)^2} \left( p - \frac{n+\alpha}{n-2} \right)  \right \}^{1/(p-1)}. 
$$
\vskip0.10in

When $\alpha =0$, Caffarelli-Gidas-Spruck \cite{C-G-S} found that every local positive solution  $u$ of \eqref{GS01} with $\frac{n}{n-2} \leq p \leq \frac{n+2}{n-2}$ is asymptotically radially symmetric
$$
u(x)=\bar{u}(|x|)(1 + O(|x|)) ~~~~~~~~  \textmd{as} ~ x \to 0, 
$$
where $\bar{u}(|x|)=\frac{1}{|\mathbb{S}^{n-1}|}\int_{\mathbb{S}^{n-1}} u(|x|\omega) d \omega$ is the spherical average of $u$.  With the help of this asymptotic radial symmetry,  they used the classical ODEs analysis to obtain the precise behavior of positive solutions near the singularity of \eqref{GS01} when  $\alpha =0$ and $\frac{n}{n-2} \leq p \leq \frac{n+2}{n-2}$.   

\vskip0.10in

In \cite{Li},   Li  proved  the asymptotic  radial symmetry of positive solutions of \eqref{GS01} with 
$$ 
-2 < \alpha \leq 0 ~~~~~~~~  \textmd{and} ~~~~~~~~1 < p\leq \frac{n+2+\alpha}{n-2}. 
$$
For other cases of $\alpha$ and $p$,  the asymptotic behavior  of  singular positive solutions of  \eqref{GS01} has also been understood very well,  see Brezis-Lions \cite{BL} and Lions \cite{L} for $\alpha =0$ and $1 < p< \frac{n}{n-2}$,  Zhang-Zhao \cite{ZZ} for $-2 < \al <2$ and $1 < p <\frac{n+\al}{n-2}$, Aviles \cite{A} for $-2 < \alpha < 2$ and $p=\frac{n+\alpha}{n-2}$,  Korevaar-Mazzeo-Pacard-Schoen \cite{K-M-P-S} for $\alpha =0$ and $p=\frac{n+2}{n-2}$, and Bidaut-V\'{e}ron and V\'{e}ron \cite{B-V} for $\alpha =0$ and $p > \frac{n+2}{n-2}$.

\vskip0.10in

In recent years, there has been an increasing interest in the study of equations involving a nonlocal diffusion operator, especially,  the fractional Laplacian, motivated by models of diverse physical phenomena such as anomalous diffusion and quasi-geostrophic flows \cite{BG,CV} and by applications in conformal geometry \cite{Ch-G,G-M-S,G-Q,GZ}.   Specially,  the following type of fractional equation       
\begin{equation}\label{Iso1} 
(-\Delta)^\sigma  u= |x|^\alpha u^p~~~~~~~~~~~\textmd{in}~ B_1 \backslash  \{0\}
\end{equation}
 has received great interest  and has been widely studied in  \cite{A-WCV,A-W2018,C-J-S-X,C-L-L,CLO,D-D-W,D-d-G-W,G-M-S,G-Q,J-d-S-X,J-L-X,LB1,LB2,YZ1,YZ2} and  references therein.   Here $\sigma \in (0, 1)$ and the fractional Laplacian operator $(-\Delta)^\sg$ is defined as
\begin{equation}\label{Def-F} 
(-\Delta)^\sg u(x) = c_{n, \sg} P.V. \int_{\mathbb{R}^n} \frac{u(x) - u(\xi)}{ |x-\xi|^{n+2\sg} } d\xi, 
\end{equation}
where $c_{n, \sg}$ is a normalization constant depending only on $n$ and $\sg$ and $P.V.$ stands for the Cauchy principal value.  In particular, when  $\alpha=0$ and $p=\frac{n+2\sigma}{n-2\sigma}$, the aforementioned asymptotic symmetry result of Caffarelli-Gidas-Spruck has been generalized to  the fractional setting by Caffarelli, Jin, Sire and Xiong \cite{C-J-S-X}.  More precisely,  the authors in \cite{C-J-S-X}  classified isolated singularities of positive solutions of \eqref{Iso1} and showed that every local positive solution $u$ of \eqref{Iso1} is asymptotically radially symmetric  
$$
u(x)=\bar{u}(|x|)(1 + O(|x|)) ~~~~~~~~  \textmd{as} ~ x \to 0, 
$$
where $\bar{u}(|x|)$ is the spherical average of $u$.  Li-Bao \cite{LB1} extended this asymptotic  radial  symmetry of positive solutions to the equation \eqref{Iso1} with   
$$
-2 \sigma< \alpha \leq 0 ~~~~~~~~  \textmd{and} ~~~~~~~~ \frac{n+\alpha}{n-2\sigma} < p \leq \frac{n+2\sigma+2\alpha}{n-2\sigma}.
$$
However, since the classical ODEs analysis is a missing ingredient in the fractional setting  to further analyze the   solutions of \eqref{Iso1} compared to the case when $\sigma =1$, the precise asymptotic behavior of positive  solutions near the singularity to  the fractional equation \eqref{Iso1}  remains as an open question.  In the recent papers \cite{YZ1,YZ2},  we established a monotonicity formula to classify isolated singularities and prove the precise asymptotic behavior of solutions to the fractional Lane-Emden equation (\eqref{Iso1} with $\al=0$) when $\frac{n}{n-2\sigma} < p <   \frac{n+2\sigma}{n-2\sigma} $.    We also refer to Fall-Felli \cite{F-F} for the precise asymptotic behavior of solutions to fractional elliptic equations with Hardy type potentials.    

\vskip0.10in 

One of the goals of this paper is to  describe  the precise asymptotic behavior of positive  solutions near the  singularity to the problem \eqref{Iso1} with Hardy weights ($-2 \sigma< \alpha <  0$) when 
$$
 \frac{n+\alpha}{n-2\sigma} < p \leq   \frac{n+2\sigma+\alpha}{n-2\sigma} ~~~~~~ \textmd{and}  ~~~~~~p \neq \frac{n+ 2\sg +2\al}{n- 2\sg}. 
$$
One motivation for studying singular solutions of \eqref{Iso1} with Hardy weights comes from the study of asymptotic behavior at  infinity of solutions of the fractional Lane-Emden equation, which can be reduced to   the similar problem for solutions near the origin  of \eqref{Iso1} with Hardy weights  via the Kelvin transformation.   We assume that $u\in C^2( B_1 \backslash  \{0\} )$ and 
$$
u\in \mathcal{L}_\sg (\mathbb{R}^n): = \left\{u \in L_{\textmd{loc}}^1 (\mathbb{R}^n) : \int_{\mathbb{R}^n} \frac{|u(x)|}{(1+ |x|)^{n+2\sg} } dx < +\infty \right\},
$$
then $(-\Delta)^\sg u(x)$ is well-defined at every point $x\in B_1 \backslash  \{0\}$. Our first main result is the following precise behavior of singular   solutions of \eqref{Iso1}.  

\begin{theorem}\label{AB-T1}
Assume  $n \geq 2$. Let $u \in C^2(B_1 \backslash \{0\}) \cap \mathcal{L}_\sigma(\mathbb{R}^n)$ be a positive solution of \eqref{Iso1} with $-2 \sigma< \alpha \leq 0 $, $\frac{n+\alpha}{n-2\sigma} < p \leq  \frac{n+2\sigma+\alpha}{n-2\sigma}$ and $p \neq \frac{n+ 2\sg +2\al}{n- 2\sg}$. 
Then either the singularity at $x=0$ is removable, or 
\be\label{Beh-0D}
|x|^{(2\sigma +\alpha)/(p-1)} u(x) \to C_{p,\sigma,\alpha} ~~~~~~~~  \textmd{as} ~ x \to 0, 
\ee
where 
\begin{equation}\label{A}
C_{p,\sigma,\alpha}=\left\{\Lambda\left(\frac{n-2\sg}{2} - \frac{2\sg + \alpha}{p-1}\right)\right\}^{\frac{1}{p-1}}
\end{equation}
and the function $\Lambda(\tau)$ is defined by 
\begin{equation}\label{AAA}
\Lambda(\tau)=2^{2\sg} \frac{\Gamma(\frac{n+2\sg+2\tau}{4})   \Gamma(\frac{n+2\sg-2\tau}{4})}{\Gamma(\frac{n-2\sg-2\tau}{4})\Gamma(\frac{n-2\sg+2\tau}{4})}.
\end{equation}
\end{theorem}

\br\label{T2R1}
If $\al < -2\sg$, then by Corollary \ref{2Co1}  Eq. \eqref{Iso1} has no positive solution in any domain $\Omega$ containing the origin. 
\er

Under the assumptions of Theorem \ref{AB-T1}, if  $u$ is a positive solution of \eqref{Iso1} with a non-removable singularity, then Theorem \ref{AB-T1} tells us  that $u$ is asymptotic to a radial solution $u^*(|x|)$ to  the same equation in  $\mathbb{R}^n \backslash \{0\}$,  where $u^*(|x|)$ is
$$
u^*(|x|) \equiv C_{p,\sigma,\alpha} |x|^{-\frac{2\sg +\al}{p-1}}. 
$$
For when $\sg =1$, Theorem \ref{AB-T1} was proved in \cite{G-S} by Gidas and Spruck. We may also see Caffarelli-Gidas-Spruck \cite{C-G-S} for the case $\sg=$1 and $\al=0$.  Unlike the proofs of \cite{C-G-S,G-S} where the ODEs analysis is an important ingredient, our proof of Theorem \ref{AB-T1} is based on a monotonicity formula, combined with  the blow up (down) arguments, the Kelvin transformation and an uniqueness result of solutions of related degenerate equations on  $\mathbb{S}^{n}_+$.  As mentioned earlier, a similar monotonicity formula for the fractional Lane-Emden equation was established  and used in our recent papers \cite{YZ1,YZ2},  where Theorem \ref{AB-T1} was obtained whenever  $\al=0 $.    We recall  the Hardy-Sobolev exponent
$$
p_S(\al):=\frac{n+2\sg+2\al}{n-2\sg}. 
$$
This exponent plays a critical role in the equation \eqref{Iso1}. When  $-2\sg < \al < 0$ and $\frac{n+\alpha}{n-2\sigma} < p < p_S(\al)$,  that is for the Hardy-Sobolev subcritical case,   our proof of Theorem \ref{AB-T1} is  similar to that in \cite{YZ1,YZ2}. Remark that, when $-2\sg < \al < 0$, Theorem \ref{AB-T1} also holds in  the Hardy-Sobolev supercritical range
$$
p_S(\al) < p \leq \frac{n+2\sg+\al}{n-2\sg}. 
$$
This is essential for applying Theorem \ref{AB-T1} to study asymptotic behavior at  infinity of solutions of the fractional Lane-Emden equation.     We emphasize that  the proof of Theorem \ref{AB-T1}  in the supercritical case is different from that in subcritical case.  One significant difference is that the  energy integral \eqref{EI000} is  non-decreasing in the subcritical case, but it is non-increasing in the supercritical case.   The other difference is that  it seems difficult to prove that every singular positive solution of \eqref{Iso1} in $\mathbb{R}^n \backslash \{0\}$  is radially symmetric  in  the supercritical case.   These differences lead us  to need some new techniques to  deal  with the supercritical case.  

\vskip0.10in

For  the Hardy-Sobolev critical case $p=p_S(\al)$ ($-2 < \al <0$) and the H\'enon's  case $0<\al <2\sg$,  we establish  the  following  classification result for isolated singularities, in particular,  it  implies  the sharp blow up estimates  of   singular solutions.        
\begin{theorem}\label{C-T2} 
Assume  $n \geq 2$. Let $u \in C^2(B_1 \backslash \{0\}) \cap \mathcal{L}_\sigma(\mathbb{R}^n)$ be a positive solution of \eqref{Iso1}.   Assume 
$$
-2 \sigma< \alpha < 2\sg ~~~~~~~~ \textmd{and} ~~~~~~~ \frac{n+\alpha}{n-2\sigma} < p <  \frac{n+2\sigma}{n-2\sigma}. 
$$
Then either the singularity at $x=0$ is removable, or there exist positive constants $C_1$ and $C_2$ such that 
\be\label{U-L}
\frac{C_1}{|x|^{(2\sg + \al)/(p-1)}} \leq u(x) \leq \frac{C_2}{|x|^{(2\sg + \al)/(p-1)}} ~~~~~~~~~  \textmd{for } ~ x \in B_{1/2} \backslash \{0\}.  
\ee
\end{theorem}

\br\label{T2R10-9275}
In Theorems  \ref{AB-T1} and \ref{C-T2},  since we do not use any special structure of the ball,  $B_1$ can be replaced by any open set containing the origin $0$.  
\er

The upper bound in \eqref{U-L} can be obtained by using  a doubling lemma of Pol\'{a}\v{c}ik-Quittner-Souplet \cite{P-Q-S}.    To derive  the lower bound  in \eqref{U-L}, one main difficulty is to prove Proposition \ref{P303} in Section 3.  In the proof of Theorem 3.3 in \cite{G-S},  Gidas and Spruck  proved the lower bound in \eqref{GSE-09} by using  the following statement: \\
"{\it  If $\liminf_{|x| \to 0} |x|^\frac{2 +\al}{p-1}u(x) =0$, then the  Harnack inequality (a Harnack inequality similar to \eqref{Har01} in this paper) implies that }
$$
\lim_{|x| \to 0} |x|^{\frac{2 +\al}{p-1}} u(x) =0."
$$
But this seems not  obvious and  requires  more explanation.  Aviles  also pointed  out  this point on p.190 in \cite{A}.  In this paper, we will make full use of a monotonicity formula  (Proposition \ref{P302})  to prove Proposition \ref{P303}. Remark that, our proof also applies to Eq.  \eqref{GS01} and thus we  could  give a  rigorous proof of  the  above statement.  We believe that the idea used here can be applied in other situations to deal with similar questions.  We also mention that Chen-Lin \cite{C-L2015} recently proved  a similar  result as Proposition \ref{P303} of this paper to a critical elliptic system by applying  Pohozaev identity, see Corollary 4.1, Lemma 4.3 and Lemma 4.4 of \cite{C-L2015}, where a spherical Harnack inequality also holds for $w_1+w_2$ but the proof is very delicate and complicated.  Our poof of Proposition \ref{P303} is also  different from the one in \cite{C-L2015}.   

\vskip0.10in

We study Eq. \eqref{Iso1} via the well known extension theorem for the fractional Laplacian $(-\Delta)^\sigma$  established by Caffarelli-Silvestre \cite{C-S},   through which  one  can study the isolated boundary singularities of a  degenerate elliptic equation with a nonlinear Neumann boundary condition in the upper half-space $\mathbb{R}^{n+1}_+$ (see \eqref{C-S-E00} and \eqref{C-S-E022} in Section 2).   We denote $\mathcal{B}_R^+$ as the upper half-ball $\mathcal{B}_R\cap \mathbb{R}_+^{n+1}$,  $\partial^+\mathcal{B}_R^+=\partial \mathcal{B}_R^+ \cap \mathbb{R}_+^{n+1}$ as the positive part of   $\partial \mathcal{B}_R^+$,  and $\partial^0 \mathcal{B}_R^+$ as the flat part of $\partial \mathcal{B}_R^+$ which is the ball $B_R$ in $\mathbb{R}^n$.
More generally, we are concerned with  the corresponding  degenerate elliptic equation in $\mathcal{B}_1^+$   with an isolated Neumann boundary singularity    
\begin{equation}\label{Iso3-2020}
\begin{cases}
-\textmd{div}(t^{1-2\sg}  \nabla U)=0~~~~~~~~~& \textmd{in} ~ \mathcal{B}_1^+,\\
\frac{\partial U}{\partial \nu^\sg}(x, 0)=\kappa_\sg |x|^\al U^p(x, 0)~~~~~~~~& \textmd{on} ~ \partial^0\mathcal{B}_1^+ \backslash  \{0\},
\end{cases}
\end{equation}
where $\frac{\partial U}{\partial \nu^\sg}(x, 0) := -\lim_{t\rightarrow 0^+}t^{1-2\sg} \partial_t U(x, t)$, the constant $\kappa_\sg=\frac{\Gamma(1-\sg)}{2^{2\sg-1} \Gamma(\sg)}$ and  $\Gamma$ is the Gamma function. By the  extension theorem of  Caffarelli and Silvestre,  if  one  knows  the  behavior  of  the  traces
$u(x):=U(x, 0)$
of the nonnegative solutions $U(x,t)$ of \eqref{Iso3-2020} near the singularity,   then  the behavior of nonnegative solutions   of  \eqref{Iso1}  follows.       

\vskip0.10in

We say that $U$ is a  nonnegative weak solution of \eqref{Iso3-2020} if $U$ is in the weighted Sobolev space $W^{1, 2}(t^{1-2\sg}, \mathcal{B}^+_1 \backslash \overline{\mathcal{B}}^+_\epsilon)$ for every $\epsilon > 0$,  $U \geq 0$, and it  satisfies  \eqref{Iso3-2020} in the sense of distributions  away from 0,  that is,     
\begin{equation}\label{Solu}
\int_{\mathcal{B}_1^+} t^{1-2\sg} \nabla U \cdot\nabla \Phi= \kappa_\sg \int_{\partial^0\mathcal{B}_1^+} |x|^\al U^p \Phi 
\end{equation}
for  every  nonnegative  $\Phi \in C_c^\infty\left((\mathcal{B}_1^+  \cup \partial^0\mathcal{B}_1^+) \backslash  \{0\}\right)$.   It follows from the regularity results  in \cite{Cab-S,J-L-X} that $U(x, t)$ is locally H\"{o}lder continuous in $\overline{\mathcal{B}}^+_{3/4} \backslash \{0\}$.  We use capital letters, such as $X=(x, t)\in \mathbb{R}^n \times \mathbb{R}_+$,  to  denote  points in $\mathbb{R}_+^{n+1}$.  
Next,  we first  classify  the  isolated boundary  singularities of the equation \eqref{Iso3-2020}.    

\begin{theorem}\label{C-T2-2020} 
Assume  $n \geq 2$. Let $U$ be a nonnegative weak solution of \eqref{Iso3-2020}.   Assume 
$$
-2 \sigma< \alpha < 2\sg ~~~~~~~~ \textmd{and} ~~~~~~~ \frac{n+\alpha}{n-2\sigma} < p <  \frac{n+2\sigma}{n-2\sigma}. 
$$
Then either the singularity at $x=0$ is removable, i.e., $U(X)$ can be extended to a continuous solution in $\overline{\mathcal{B}_{1/2}^+}$, or there exist two positive constants $C_1$ and $C_2$ such that 
\be\label{U-L-2020}
\frac{C_1}{|X|^{(2\sg + \al)/(p-1)}} \leq U(X) \leq \frac{C_2}{|X|^{(2\sg + \al)/(p-1)}}. 
\ee
\end{theorem}

Furthermore, we  can describe  the  precise asymptotic  behavior of singular positive  solutions $U(x, t)$ of \eqref{Iso3-2020} as $(x, t) \to 0$.     

\begin{theorem}\label{AB-T1-2020}
Assume  $n \geq 2$. Let $U$ be a nonnegative weak solution of \eqref{Iso3-2020} with $-2 \sigma< \alpha \leq 0 $, $\frac{n+\alpha}{n-2\sigma} < p \leq  \frac{n+2\sigma+\alpha}{n-2\sigma}$ and $p \neq \frac{n+ 2\sg +2\al}{n- 2\sg}$. 
Then either the singularity at $x=0$ is removable, or 
\be\label{Beh-0D-2020}
U(x,t)=  C_{p,\sigma,\alpha} \int_{\mathbb{R}^n} P_\sigma(x-y, t) |y|^{- \frac{2\sigma +\alpha}{p-1}} dy   \Big(1 + o(1) \Big)  ~~~~~~~~  \textmd{as} ~ (x, t) \to 0, 
\ee
where the constant $C_{p,\sigma,\alpha}$ is given by \eqref{A}-\eqref{AAA},   $P_\sigma(x, t)$ is the Poisson  kernel 
\be\label{Poisson0814}
P_\sg(x, t)=p_{n, \sg}\frac{t^{2\sg}}{(|x|^2 + t^2)^{\frac{n+2\sg}{2}}}
\ee
and $p_{n, \sg}$ is a positive constant chosen such that $\int_{\mathbb{R}^n} P_\sg(x, 1)dx=1$.  
\end{theorem}

Remark  that Theorems \ref{C-T2-2020} and \ref{AB-T1-2020}  give us not only the  asymptotic   behavior of the trace $u(x)$ of a singular solution $U(x, t)$,  but also the  asymptotic behavior of the solution $U(x, t)$ near the boundary singularity.

\vskip0.10in
The following two theorems treat  the isolated singularities located at infinity. 
\begin{theorem}\label{C-T3}
Assume  $n \geq 2$.  Let $u \in \mathcal{L}_\sg(\mathbb{R}^n)$  be a nonnegative  $C^2$ solution of
\begin{equation}\label{Iny}
(-\Delta)^\sigma u = |x|^\al u^p~~~~~~~~~~~ \textmd{in} ~ \{x \in \mathbb{R}^n  : |x| > 1\}
\end{equation}
with $\al > -2\sg$ and $1 < p <\frac{n+2\sg}{n-2\sg}$. 
\begin{itemize}

\item [(1)] If $1 < p <\frac{n+\al}{n-2\sg}$, then necessarily  $u(x) \equiv 0$ in $\{x \in \mathbb{R}^n  : |x| > 1\}$.  

\item [(2)] If $\frac{n+\alpha}{n-2\sigma} < p <  \frac{n+2\sigma}{n-2\sigma}$,  then either the singularity at $\infty$ is removable, i.e., there exists $C >0$ such that
$$
u(x) \leq \frac{C}{|x|^{n-2\sg}}~~~~~~~~~~~~~  \textmd{near} ~~ x=\infty, 
$$
or there exist positive constants $C_1$, $C_2$ such that 
$$
\frac{C_1}{|x|^{(2\sg + \al)/(p-1)}} \leq u(x) \leq \frac{C_2}{|x|^{(2\sg + \al)/(p-1)}} ~~~~~~~~~  \textmd{near} ~~ x=\infty. 
$$
\end{itemize}
\end{theorem}

\begin{theorem}\label{C-T4}
Assume  $n \geq 2$.  Let $u \in \mathcal{L}_\sg(\mathbb{R}^n)$  be a positive $C^2$ solution of \eqref{Iny} with  $-2\sg < \al \leq 0$, $\frac{n+\al}{n-2\sg} < p  \leq \frac{n+2\sg +\al}{n-2\sg}$ and $p\neq \frac{n+2\sg +2\al}{n-2\sg}$. Then either there exists $\beta >0$ such that 
$$
\lim_{|x| \to  \infty} |x|^{n-2\sg}u(x) = \beta, 
$$
or 
$$
\lim_{|x| \to  \infty} |x|^{(2\sigma +\alpha)/(p-1)} u(x) = C_{p,\sigma,\alpha}, 
$$
where $C_{p,\sigma,\alpha}$ is given by \eqref{A}. 
\end{theorem}

In particular, we give a complete classification of isolated singularities of positive solutions to the fractional Lane-Emden equation near $\infty$ .  

\begin{corollary}\label{L-E19}
Assume  $n \geq 2$.  Let $u \in \mathcal{L}_\sg(\mathbb{R}^n)$  be a positive $C^2$ solution of 
\be\label{LM2019}
(-\Delta)^\sg u =u^p~~~~~~~~~~~~~ \textmd{in} ~ \{x \in \mathbb{R}^n  : |x| > 1\} 
\ee
with   $\frac{n}{n-2\sg} < p  <  \frac{n+2\sg }{n-2\sg}$.  Then either there exists $\beta >0$ such that 
\be\label{Fas}
\lim_{|x| \to  \infty} |x|^{n-2\sg}u(x) = \beta, 
\ee
or 
\be\label{Slow}
\lim_{|x| \to  \infty} |x|^{\frac{2\sg}{p-1}} u(x) = C_{p,\sigma,0}, 
\ee
where $C_{p,\sigma,0}$ is given by \eqref{A}. 
\end{corollary}

\begin{remark}\label{AW18}
Our characterization of  isolated singularities near $\infty$ of the fractional  Lane-Emden  equation  is complemented by the existence results
of fast-decay solutions satisfying  \eqref{Fas}  which have  been recently  constructed  by Ao-Chan-DelaTorre-Fontelos-Gonz\'alez-Wei \cite{A-WCV, A-W2018}. 
More precisely, for some exponent $p_1=p_1(n, \sg) \in (\frac{n}{n-2\sg}, \frac{n+2\sg}{n-2\sg})$,  and for every $\beta \in(0, \infty)$, there exists a positive solution of \eqref{LM2019}  satisfying  \eqref{Fas} which was  proved in \cite{A-WCV}  when $ \frac{n}{n-2\sg}< p< p_1$ and  in \cite{A-W2018} when $p_1\leq p < \frac{n+2\sg}{n-2\sg}$. 
\end{remark}

Finally, we establish an uniqueness theorem for global singular solutions.  

\begin{theorem}\label{unique-2020}
Assume  $n \geq 2$. Let $U  \in W_{loc}^{1,2}(t^{1-2\sg}, \overline{\mathbb{R}^{n+1}_+} \backslash \{0\})$ be a nonnegative weak solution of 
\begin{equation}\label{Un-2020-01}
\begin{cases}
-\textmd{div}(t^{1-2\sg}  \nabla U)=0~~~~~~~~~& \textmd{in} ~ \mathbb{R}^{n+1}_+,\\
\frac{\partial U}{\partial \nu^\sg}(x, 0)=\kappa_\sg |x|^\al U^p(x, 0)~~~~~~~~& \textmd{on} ~ \partial^0\mathbb{R}^{n+1}_+  \backslash  \{0, \infty\}
\end{cases} 
\end{equation} 
with $-2 \sigma< \alpha \leq 0 $, $\frac{n+\alpha}{n-2\sigma} < p \leq  \frac{n+2\sigma+\alpha}{n-2\sigma}$ and $p \neq \frac{n+ 2\sg +2\al}{n- 2\sg}$.  Assume that the  two isolated singularities of the trace $u(x)$ of $U(x, t)$ at $x=0$ and $x=\infty$ are non-removable. Then necessarily we have 
\begin{equation}\label{Un-2020-02}
U(x, t)=   C_{p,\sigma,\alpha} \int_{\mathbb{R}^n} P_\sigma(x-y, t) |y|^{- \frac{2\sigma +\alpha}{p-1}} dy, 
\end{equation}
where the constant $C_{p,\sigma,\alpha}$ is given by \eqref{A}-\eqref{AAA} and $P_\sigma(x, t)$ is the Poisson kernel.  
\end{theorem}

\vskip0.10in

The rest of  this  paper is organized as follows. In Section 2,  we introduce the extension formulation  for $(-\Delta)^\sg$  established by  Caffarelli-Silvestre \cite{C-S} and  provide   some  a priori  estimates.   In Section 3, we establish an important monotonicity formula and prove Theorems \ref{C-T2} and  \ref{C-T2-2020}. In Section 4,  we show the precise asymptotic behavior of singular  solutions stated in Theorems \ref{AB-T1} and \ref{AB-T1-2020} and also give the proof of Theorem \ref{unique-2020}.     In Section 5,  we prove Theorems \ref{C-T3} and \ref{C-T4}.

\section{Preliminaries}
In this section, we introduce some notations and prove  some important estimates  which will be used in this paper.  

\vskip0.10in 

We use capital letters, such as $X=(x, t)\in \mathbb{R}^n \times \mathbb{R}_+$,  to  denote  points in $\mathbb{R}_+^{n+1}$.  We denote $\mathcal{B}_R$ as the ball in $\mathbb{R}^{n+1}$ with radius $R$ and center 0,  and $B_R$ as the ball in $\mathbb{R}^n$ with radius $R$ and center $0$. We also denote $\mathcal{B}_R^+$ as the upper half-ball $\mathcal{B}_R\cap \mathbb{R}_+^{n+1}$,  $\partial^+\mathcal{B}_R^+=\partial \mathcal{B}_R^+ \cap \mathbb{R}_+^{n+1}$ as the positive part of   $\partial \mathcal{B}_R^+$,  and $\partial^0 \mathcal{B}_R^+$ as the flat part of $\partial \mathcal{B}_R^+$ which is the ball $B_R$ in $\mathbb{R}^n$.  For a more general domain $\Omega \subset \mathbb{R}^{n+1}_+$, we  denote $\partial^0 \Omega$ as the interior of $\overline{\Omega} \cap \partial \mathbb{R}^{n+1}_+$ in $\mathbb{R}^n$. 

\vskip0.10in

As mentioned before,  we will study the fractional Hardy-H\'enon equation \eqref{Iso1} via the well known extension theorem for the fractional Laplacian $(-\Delta)^\sg$ established by Caffarelli-Silvestre \cite{C-S}.  Assume $u \in C^2(B_1 \backslash  \{0\}) \cap \mathcal{L}_\sg (\mathbb{R}^n)$. For $X=(x, t) \in \mathbb{R}_+^{n+1}$, let 
\begin{equation}\label{Extension}
U(x, t)=\int_{\mathbb{R}^n} P_\sg (x-y, t) u(y) dy,
\end{equation}
where 
\begin{equation}\label{Poss095}
P_\sg(x, t)=p_{n, \sg}\frac{t^{2\sg}}{(|x|^2 + t^2)^{\frac{n+2\sg}{2}}}
\end{equation} 
and $p_{n, \sg}$ is a positive constant chosen such that $\int_{\mathbb{R}^n} P_\sg(x, 1)dx=1$.  Then  $U\in C^2 (\mathbb{R}_+^{n+1}) \cap C 
\left( (\mathcal{B}_1^+ \cup \partial^0 \mathcal{B}_1^+) \backslash  \{0\} \right)$,  $t^{1-2\sg} \partial_t U \in C 
\left( (\mathcal{B}_1^+ \cup \partial^0 \mathcal{B}_1^+) \backslash  \{0\} \right)$ and 
\begin{equation}\label{C-S-E00}
\begin{cases}
-\textmd{div}(t^{1-2\sg}  \nabla U)=0~~~~~~~~~& \textmd{in} ~ \mathbb{R}_+^{n+1},\\
U(x, 0)=u(x)~~~~~~~~& \textmd{on} ~ \partial^0\mathcal{B}_1^+ \backslash  \{0\}.\\
\end{cases}
\end{equation}
By the extension formulation in \cite{C-S},  we have 
\begin{equation}\label{C-S-E022}
\frac{\partial U}{\partial \nu^\sg}(x, 0)  = \kappa_\sg (-\Delta)^\sg u(x) ~~~~~~~~ \textmd{on} ~ \partial^0\mathcal{B}_1^+ \backslash  \{0\}, 
\end{equation}
where $\frac{\partial U}{\partial \nu^\sg}(x, 0) := -\lim_{t\rightarrow 0^+}t^{1-2\sg} \partial_t U(x, t)$, the constant $\kappa_\sg=\frac{\Gamma(1-\sg)}{2^{2\sg-1} \Gamma(\sg)}$ and $\Gamma$ is the Gamma function. 

\vskip0.1in
 
Instead of Eq.  \eqref{Iso1} we may study the following degenerate elliptic equation with an isolated  Neumann boundary singularity  
\begin{equation}\label{Iso3} 
\begin{cases}
-\textmd{div}(t^{1-2\sg}  \nabla U)=0~~~~~~~~~& \textmd{in} ~ \mathcal{B}_1^+,\\
\frac{\partial U}{\partial \nu^\sg}(x, 0)=\kappa_\sg |x|^\al U^p(x, 0)~~~~~~~~& \textmd{on} ~ \partial^0\mathcal{B}_1^+ \backslash  \{0\}. 
\end{cases}
\end{equation}
By \eqref{C-S-E00} and \eqref{C-S-E022}, the asymptotic behavior of  solutions  near the singularity of  \eqref{Iso1} can be obtained  from that of the traces $u(x):=U(x, 0)$ of  the solutions $U(x,t)$ of \eqref{Iso3}.  

\vskip0.1in 

We recall  that $U$ is a  nonnegative weak solution of \eqref{Iso3} if $U$ is in the weighted Sobolev space $W^{1, 2}(t^{1-2\sg}, \mathcal{B}^+_1 \backslash \overline{\mathcal{B}}^+_\epsilon)$ for every $\epsilon > 0$,  $U \geq 0$, and it  satisfies  \eqref{Iso3} in the sense of distributions  away from 0,  that is,    
\begin{equation}\label{Solu}
\int_{\mathcal{B}_1^+} t^{1-2\sg} \nabla U \cdot\nabla \Phi= \kappa_\sg \int_{\partial^0\mathcal{B}_1^+} |x|^\al U^p \Phi 
\end{equation}
for every  nonnegative  $\Phi \in C_c^\infty\left((\mathcal{B}_1^+  \cup \partial^0\mathcal{B}_1^+) \backslash  \{0\}\right)$.

\vskip0.10in

We say that the origin $0$  is a removable singularity of solution $U$ of \eqref{Iso3} if $U(x,t)$ can be extended as a continuous function near the origin, otherwise we say that the origin $0$  is a non-removable singularity.   

\vskip0.10in

We say $U \in W_{loc}^{1,2}(t^{1-2\sg}, \overline{\mathbb{R}^{n+1}_+})$ if $U \in W^{1,2}(t^{1-2\sg}, \mathcal{B}_R^+)$ for all $R>0$, and we say $U \in W_{loc}^{1,2}(t^{1-2\sg}, \overline{\mathbb{R}^{n+1}_+} \backslash \{0\})$ if $U \in W^{1,2}(t^{1-2\sg}, \mathcal{B}_R^+ \backslash \overline{\mathcal{B}}_\epsilon^+)$ for all $R > \epsilon >0$.

\vskip0.10in

We now establish the basic singularity and decay estimates.  In the case $\sg =1$, that is for the Laplacian, the corresponding results  were  proved in \cite{G-S,Soup}.   
\begin{proposition}\label{SDE}
Let $n \geq 2$,  $\al \in \mathbb{R}$ and $1 < p < \frac{n+2\sg}{n-2\sg}$. 

\begin{itemize}

\item [(1)] Suppose that $U$ is a nonnegative weak solution of \eqref{Iso3}.  Then there exists a constant $C=C(n,p,\al,\sg)$ such that
\be\label{1se01}
U(x, 0) \leq  \frac{C}{|x|^{(2\sg +\al) / (p-1)}},  ~~~~~~~~ 0 < |x| < \frac{1}{2}. 
\ee

\item [(2)] Suppose that $U$ is a nonnegative weak solution of 
\begin{equation}\label{Iso3E}
\begin{cases}
-\textmd{div}(t^{1-2\sg}  \nabla U)=0~~~~~~~~~& \textmd{in} ~ \mathbb{R}_{+}^{n+1} \backslash \overline{\mathcal{B}_1^+},\\
\frac{\partial U}{\partial \nu^\sg}(x, 0)=\kappa_\sg |x|^\al U^p(x, 0)~~~~~~~~& \textmd{on} ~ B_1^c, 
\end{cases}
\end{equation}
where $B_1^c:=\{x \in \mathbb{R}^n : |x| >1 \}$. Then there exists a constant $C=C(n,p,\al,\sg)$ such that
\be\label{1de01}
U(x,0) \leq \frac{C}{|x|^{(2\sg +\al) / (p-1)}},  ~~~~~~~~ |x| > 2. 
\ee
\end{itemize}
\end{proposition}

To prove Proposition \ref{SDE},  we need the following lemma.

\bl\label{LM-01} 
Let $n \geq 2$ and $1 < p < \frac{n+2\sg}{n-2\sg}$. Let $K \in C^1(\overline{B_1})$ satisfy
\begin{equation}\label{K-01}
\|K\|_{C^1(\overline{B_1})} \leq C_1 ~~~~~~ \textmd{and}  ~~~~~~ K(x) \geq C_2,~~ x \in \overline{B_1}, 
\end{equation}
for some  constants $C_1, C_2 >0$. Suppose that $U$ is a nonnegative weak solution of 
\begin{equation}\label{K90}
\begin{cases}
-\textmd{div}(t^{1-2\sg}  \nabla U)=0~~~~~~~~~& \textmd{in} ~ \mathcal{B}_1^+,\\
\frac{\partial U}{\partial \nu^\sg}(x, 0)=K(x) U^p(x, 0)~~~~~~~~& \textmd{on} ~ \partial^0\mathcal{B}_1^+. 
\end{cases}
\end{equation}
Then there exists a constant $C$, depending only on $n, \sg, \gamma, p, C_1, C_2$, such that 
$$
U(x,0) \leq C  \left[  \textmd{dist}(x, \partial B_1)  \right]^{-\frac{2\sg}{p-1}} ,~~~~~~~~ x\in B_1. 
$$
\el

\bp
Suppose by contradiction that there exists a sequence of solutions $U_i$ of \eqref{K90} and a sequence of points $x_i \in B_1$ such that
$$
M_i(x_i) \textmd{dist}(x_i, \partial B_1) >  2 i,~~~~~~~ i=1,2,\cdots, 
$$
where the functions $M_i$ are defined by
$$
M_i(x) = ( U_i(x ,0) )^{\frac{p-1}{2\sg}},~~~~~~~ x \in B_1. 
$$
By the doubling lemma of Pol\'{a}\v{c}ik-Quittner-Souplet \cite{P-Q-S},  there exists another sequence $y_i \in B_1$ such that
\begin{equation}\label{PQS78}
M_i(y_i) \textmd{dist}(y_i, \partial B_1) > 2i, ~~~~~~~ M_i(y_i) \geq M_i(x_i)
\end{equation}
and
\begin{equation}\label{PQS39}
M_i(z) \leq 2 M_i(y_i)~~~~~~~~ \textmd{for} ~ \textmd{any} ~ |z - y_i| \leq i \lambda_i, 
\end{equation}
where $\lambda_i:=M_i(y_i)^{-1}$. Note that $\lambda_i \to 0$ as $i\to \infty$.  We now define 
$$
\bar{U}_i(x, t) = \lambda_i^{\frac{2\sg}{p-1}} U_i(y_i + \lambda_i x, \lambda_i t),~~~~~~~ (x, t)\in \Omega_i 
$$
with 
$$
\Omega_i =\left\{(x, t)\in \mathbb{R}^{n+1}_+  :    (y_i + \lambda_i x, \lambda_i t)  \in \mathcal{B}_1^+ \backslash \{0\}\right\}.
$$
Then $\bar{U}_i$ satisfies $\bar{U}_i(0)=1$ and
\begin{equation}\label{Upp04}
\begin{cases}
-\textmd{div}(t^{1-2\sg}  \nabla \bar{U}_i)=0~~~~~~~~~& \textmd{in} ~\Omega_i,\\
\frac{\partial \bar{U}_i}{\partial \nu^\sg}(x, 0)=\bar{K}_i(x)  \bar{U}_i(x, 0)^p~~~~~~~~& \textmd{on} ~\partial^0\Omega_i, 
\end{cases}
\end{equation}
where $\bar{K}_i(x) = K(y_i + \lambda_i x)$ for $x \in \partial^0\Omega_i$. Moreover, by \eqref{PQS39}  we have 
$$
\bar{U}_i(x, 0) \leq 2^{\frac{2\sg}{p-1}},~~~~~~~~~ x \in B_i(0) \subset \mathbb{R}^n. 
$$

On the other hand, by \eqref{K-01} we  know  that $C_2 \leq \bar{K}_i(x) \leq C_1$ and, for each $R >0$ and $i \geq i_0(R)$ large enough,
\begin{equation}\label{K-030}
\|\bar{K}_i\|_{C^1(\overline{B_R})} \leq C_1
\end{equation}
and 
\begin{equation}\label{K-08}
|\bar{K}_i(y) - \bar{K}_i(z)| \leq C_1|\lambda_i(y - z)| \leq C_1|y - z|,~~~~~ y, z \in B_R. 
\end{equation}
Therefore, by Arzela-Ascoli's theorem, there exists $\bar{K} \in C(\mathbb{R}^n)$ such that, after extracting a subsequence, $\bar{K}_i \to \bar{K}$ in $C_{loc}(\mathbb{R}^n)$. Moreover, from  \eqref{K-08} we have for any $y , z \in \mathbb{R}^n$ that 
$$
|\bar{K}_i(y) - \bar{K}_i(z)|  \to 0 ~~~~~~~ \textmd{as} ~ i \to \infty, 
$$
and hence the function $\bar{K}$ is actually a constant $K_0 >0$.   

\vskip0.10in

It follows from Corollary 2.10  and  Theorem 2.15 of Jin-Li-Xiong \cite{J-L-X} that there exists $\gamma \in (0, 1)$ such that for every $R>1$,
$$
\|\bar{U}_i\|_{W^{1,2}(t^{1-2\sg}, \mathcal{B}_R^+)} + \|\bar{U}_i\|_{C^\gamma(\overline{\mathcal{B}}_R^+)} \leq C(R),
$$
where $C(R)$ is independent of $i$.  Thus,  there is a subsequence of $i \rightarrow \infty$, still denoted by itself,  and a  function $\bar{U}\in W_{loc}^{1,2}(t^{1-2\sg}, \overline{\mathbb{R}^{n+1}_+}) \cap C_{loc}^\gamma(\overline{\mathbb{R}^{n+1}_+})$ such that as $ i \rightarrow \infty$,
$$
\begin{cases}
\bar{U}_i \rightharpoonup \bar{U}~~~~~~~\textmd{weakly} ~ \textmd{in} ~~~~~~~~~~&W_{loc}^{1,2}(t^{1-2\sg}, \overline{\mathbb{R}^{n+1}_+})  ,\\
\bar{U}_i \rightarrow \bar{U}~~~~~~~  \textmd{in}~~~~~~~&C_{loc}^{\gamma/2}(\overline{\mathbb{R}^{n+1}_+}). 
\end{cases}
$$
Moreover, $\bar{U}$ is a  nonnegative solution of 
\begin{equation}\label{Upp051}
\begin{cases}
-\textmd{div}(t^{1-2\sg}  \nabla \bar{U})=0~~~~~~~~~& \textmd{in} ~\mathbb{R}^{n+1}_+,\\
\frac{\partial \bar{U}}{\partial \nu^\sg}(x, 0)=K_0 \bar{U}^p(x, 0)~~~~~~~~& \textmd{on} ~\mathbb{R}^n,
\end{cases}
\end{equation}
and $\bar{U}(0)=1$. Since $p < \frac{n+2\sg}{n-2\sg}$, this contradicts  the Liouville type theorem in \cite{J-L-X} (See Theorem 1.8 and Remark 1.9  in  \cite{J-L-X} ).    
\ep

\vskip0.10in

\noindent{\it Proof of Proposition \ref{SDE}. }
Suppose either $\Omega =\{x \in \mathbb{R}^n : 0 < |x| < 1\}$ and $0 < |x_0| < \frac{1}{2}$, or $\Omega =\{x \in \mathbb{R}^n :  |x| > 1\}$ and $ |x_0| > 2$.    Take 
$$
\lambda=\frac{|x_0|}{2}.
$$
Then, for any $y \in B_1$, we have $\frac{|x_0|}{2} < |x_0 + \lambda y| < \frac{3|x_0|}{2}$. Hence $x_0 + \lambda y \in \Omega$ in either case. Define 
$$
W(y, t)= \lambda^{\frac{2\sg + \al}{p-1}} U(x_0 + \lambda y, \lambda t). 
$$
Then $W$ is a nonnegative solution of 
$$
\begin{cases}
-\textmd{div}(t^{1-2\sg}  \nabla W)=0~~~~~~~~~& \textmd{in} ~ \mathcal{B}_1^+,\\
\frac{\partial W}{\partial \nu^\sg}(y, 0)=K(y) W^p(y, 0)~~~~~~~~& \textmd{on} ~ \partial^0\mathcal{B}_1^+, 
\end{cases}
$$
where $K(y)=|y + \frac{x_0}{\lambda}|^\al$ for $y \in B_1$. Clearly 
$$
1 \leq |y + \frac{x_0}{\lambda}| \leq 3 ~~~~~~~~~ \textmd{for} ~  \textmd{all} ~  y \in \overline{B_1}. 
$$
Therefore,  $\|K\|_{C^1(\overline{B_1})} \leq C_1(\al)$ and $K(y) \geq C_2(\al)$ in $\overline{B_1}$ for some constants $C_1(\al), C_2(\al) >0$. By Lemma \ref{LM-01}  we obtain $ W(0) \leq C $. This implies that
$$
U(x_0, 0) \leq C \lambda^{-\frac{2\sg + \al}{p-1}} \leq C |x_0|^{-\frac{2\sg + \al}{p-1}}. 
$$
The desired conclusion follows. 
\hfill$\square$

\vskip0.10in

\begin{corollary}\label{2Co1}
Let $n \geq 2$ and $1 < p < \frac{n+2\sg}{n-2\sg}$.  Suppose that $U$ is a nonnegative weak solution of \eqref{Iso3}. If  $\al <-2\sg$, then $U(x) \equiv 0$ in $( \mathcal{B}_1^+ \cup \partial^0 \mathcal{B}_1^+) \backslash \{0\}$. 
\end{corollary}
\bp
Since $\al <-2\sg$,  by \eqref{1se01}  we know 
\be\label{as0}
U(x, 0) \to 0 ~~~~~~~~~~ \textmd{as} ~ x \to 0.
\ee
Assume by contradiction that there exists $X_0 \in ( \mathcal{B}_1^+ \cup \partial^0 \mathcal{B}_1^+) \backslash \{0\}$ such that $U(X_0) > 0$. Then the maximum principle implies that 
$$ 
U(X) >0 ~~~~~~~~~~ \textmd{for} ~  \textmd{all} ~ X \in ( \mathcal{B}_1^+ \cup \partial^0 \mathcal{B}_1^+) \backslash \{0\}.
$$
By Proposition 3.1 in \cite{J-L-X}, we have
$$
\liminf_{X \to 0} U(X) >0, 
$$
a contradiction with \eqref{as0}.   
\ep

Now we recall a Harnack inequality. For its proof,  see  \cite{Cab-S,J-L-X}.
\begin{lemma}\label{HI01}
Let $U \in W^{1,2}(t^{1-2\sg}, \mathcal{B}^+_1)$ be a nonnegative weak solution of
\begin{equation}\label{P11}
\begin{cases}
-\textmd{div}(t^{1-2\sg}  \nabla U)=0~~~~~~~~~& \textmd{in} ~\mathcal{B}^+_1,\\
\frac{\partial U}{\partial \nu^\sg}(x, 0)=a(x)U(x, 0)~~~~~~~~& \textmd{on} ~\partial^0\mathcal{B}^+_1. 
\end{cases}
\end{equation}
If $a\in L^q(B_1)$ for some $q > \frac{n}{2\sg}$, then we have
\be\label{P21}
\sup_{\mathcal{B}^+_{1/2}} U \leq C  \inf_{\mathcal{B}^+_{1/2}} U,
\ee
where $C$ depends only on $n, \sg$ and $\|a\|_{L^q(B_1)}$. 
\end{lemma}

One very useful consequence of Proposition \ref{SDE} is the following Harnack inequality.  
\bl\label{LM-02} 
Let $n \geq 2$,  $\al \in \mathbb{R}$ and $1 < p < \frac{n+2\sg}{n-2\sg}$. 

\begin{itemize}

\item [(1)] Suppose that $U$ is a nonnegative weak solution of \eqref{Iso3}.  Then there exists a constant $C=C(n,p,\al,\sg)$ such that for all $0 < r < \frac{1}{8}$, we have 
\be\label{Har01}
\sup_{\mathcal{B}_{2r}^+ \backslash \overline{\mathcal{B}_{r/2}^+}} U \leq C  \inf_{\mathcal{B}_{2r}^+ \backslash \overline{\mathcal{B}_{r/2}^+}} U. 
\ee

\item [(2)] Suppose that $U$ is a nonnegative weak solution of \eqref{Iso3E}. Then there exists a constant $C=C(n,p,\al,\sg)$ such that for all $r > 8$, we have 
\be\label{Har02}
\sup_{\mathcal{B}_{2r}^+ \backslash \overline{\mathcal{B}_{r/2}^+}} U \leq C  \inf_{\mathcal{B}_{2r}^+ \backslash \overline{\mathcal{B}_{r/2}^+}} U. 
\ee
\end{itemize}
\el
\bp
Let 
$$
V_r(X)=U(rX)
$$
for $X \in \mathcal{B}_4^+ \backslash \overline{\mathcal{B}_{1/4}^+}$.  Then $V_r$ satisfies 
\begin{equation}\label{Upp061}
\begin{cases}
-\textmd{div}(t^{1-2\sg}  \nabla V_r)=0~~~~~~~~~& \textmd{in} ~\mathcal{B}_4^+ \backslash \overline{\mathcal{B}_{1/4}^+},\\
\frac{\partial V_r}{\partial \nu^\sg}(x, 0)=a_r(x)v_r(x)~~~~~~~~& \textmd{on} ~B_4 \backslash \overline{B}_{1/4},\\
\end{cases}
\end{equation}
where $v_r(x)=V_r(x,0)$ and $a_r(x)=r^{2\sg + \al} |x|^\al \left( u(rx) \right)^{p-1}$.  By Proposition \ref{SDE}, 
$$
|a_r(x)| \leq C ~~~~~~~\textmd{for} ~ \text{all} ~1/4 \leq |x| \leq 4,
$$
where $C$ is a positive constant independent of $r$ and $U$.   By  Harnack inequality in Lemma  \ref{HI01} and the standard Harnack inequality for uniformly elliptic equations, we have
$$
\sup_{\frac{1}{2} \leq |X| \leq 2} V_r(X)  \leq C \inf_{\frac{1}{2}\leq |X| \leq 2} V_r(X),
$$
where $C$ is another positive constant independent of $r$ and $U$.  We complete the proof by rescaling back to $U$.  
\end{proof}

\section{Classification of Isolated Singularities at $x=0$}
In this section, we classify the  isolated singularities of  positive  solutions of \eqref{Iso3} near the origin.    To this end, we need to establish a monotonicity formula for the nonnegative solutions $U$  of \eqref{Iso3} ({\it resp.}  of \eqref{Iso3E}).  Let  $U$  be a nonnegative solution of \eqref{Iso3} ({\it resp.}  of  \eqref{Iso3E}),  we  define
\be\label{EI000}
\aligned
E(r;U):= &  r^{\frac{2(p+1)\sg + 2\al}{p-1} - n }\left[ r \int_{\partial^+\mathcal{B}_r^+} t^{1-2\sg} | \frac{\partial U}{\partial \nu} |^2 + \frac{2\sg + \al}{p-1} \int_{\partial^+\mathcal{B}_r^+} t^{1-2\sg}\frac{\partial U}{\partial \nu} U\right] \\
& + \frac{2\sg + \al}{p-1}\left(\frac{2 \sg + \al}{p-1} - \frac{ n - 2\sg}{2} \right) r^{\frac{2(p+1)\sg +2\al }{p-1} - n -1}  \int_{\partial^+\mathcal{B}_r^+} t^{1-2\sg} U^2  \\
& -  \frac{1}{2} r^{\frac{2(p+1)\sg + 2\al}{p-1} - n+1 }  \int_{\partial^+\mathcal{B}_r^+} t^{1-2\sg} |\nabla U|^2  \\
& +  \frac{\kappa_\sg}{p+1} r^{\frac{(2\sg +\al) (p+1)}{p-1} -n +1}\int_{\partial B_r} u^{p+1} .
\endaligned
\ee
We recall that the  Hardy-Sobolev critical exponent is defined by  
$$
 p_S(\al):= \frac{n+2\sg +2\al}{n-2\sg}. 
$$
Then, we have the following monotonicity formula.  
\begin{proposition}\label{P302} 

Let $n \geq 2$,  $\al \in \mathbb{R}$ and $1 < p < \frac{n+2\sg}{n-2\sg}$. 

\begin{itemize}

\item [(1)] Suppose that $p  \leq  p_S(\al) $ and  $U$ is a nonnegative weak solution of \eqref{Iso3} ({\it resp.}  of \eqref{Iso3E}).  Then  $E(r; U)$ is non-decreasing in $r \in (0, 1)$ ({\it resp.} in $r \in (1, \infty)$).  Moreover, 
$$
\frac{d}{dr} E(r; U)=J_1 r^{\frac{2(p+1)\sg + 2\al}{p-1} -n} \int_{\partial^+\mathcal{B}_r^+} t^{1-2\sg} \left(\frac{\partial U}{\partial \nu} + \frac{2\sg + \al }{p-1} \frac{U}{r}\right)^2,
$$
where $J_1=\frac{n-2\sg}{p-1}\left( \frac{n+2\sg +2\al}{n-2\sg} - p \right) \geq 0$.  

\item [(2)]  Suppose that  $p > p_S(\al)$ and  $U$ is a nonnegative weak solution of \eqref{Iso3} ({\it resp.}  of \eqref{Iso3E}).  Then  $E(r; U)$ is non-increasing in $r \in (0, 1)$ ({\it resp.} in $r \in (1, \infty)$).  Moreover, 
$$
\frac{d}{dr} E(r; U)=J_1 r^{\frac{2(p+1)\sg +2\al}{p-1} -n} \int_{\partial^+\mathcal{B}_r^+} t^{1-2\sg} \left(\frac{\partial U}{\partial \nu} + \frac{2\sg +\al }{p-1} \frac{U}{r}\right)^2,
$$
where $J_1=\frac{n-2\sg}{p-1}\left( \frac{n+2\sg +2\al}{n-2\sg} - p \right) <0 $.  
\end{itemize} 
\end{proposition}

\bp
We shall take the standard polar coordinates in $\mathbb{R}^{n+1}_+$:  $X=(x,t)=r\theta$, where $r=|X|$ and $\theta=\frac{X}{|X|}$.  Let $\theta_1=\frac{t}{|X|}$ denote the component of $\theta$ in the $t$ direction and 
$$\mathbb{S}^n_+=\{X\in \mathbb{R}^{n+1}_+ : r=1, \theta_1 >0\}$$
denote the upper unit half-sphere.
 
 \vskip0.10in
 
Let $U$ be a nonnegative weak solution of \eqref{Iso3}.  Using  the classical change of variable in Fowler \cite{F},
$$
V(s, \theta)= r^{\frac{2\sg +\al}{p-1}} U(r, \theta),~~~~~ s=\ln r.
$$
Direct  calculations show that $V$ satisfies
\begin{equation}\label{P30111}
\begin{cases}
V_{ss} - J_1 V_s - J_2V + \theta_1^{2\sg-1} \text{div}_\theta (\theta_1^{1-2\sg} \nabla_\theta V)=0~~~~~~~~~& \textmd{in} ~(-\infty, 0) \times\mathbb{S}^n_+ ,\\
-\lim_{\theta_1\rightarrow 0^+} \theta_1^{1-2\sg} \partial_{\theta_1} V = \kappa_\sg V^p~~~~~~~~& \textmd{on} ~(-\infty, 0) \times \partial \mathbb{S}^n_+,\\
\end{cases}
\end{equation}
where 
$$
J_1=\frac{n - 2\sg}{p-1}\left( \frac{n+2\sg +2\al}{n -2\sg} - p \right),~~~~~
J_2=\frac{2\sg +\al}{p-1}\left(n-2\sg-\frac{2\sg + \al}{p-1}\right).
$$
Multiplying \eqref{P30111} by $V_s$ and integrating,  we have
\begin{equation}\label{P30122}
\aligned
& \int_{\mathbb{S}^n_+}\theta_1^{1-2\sg}V_{ss}V_s   - J_2\int_{\mathbb{S}^n_+}\theta_1^{1-2\sg}VV_s -\int_{\mathbb{S}^n_+}\theta_1^{1-2\sg}\nabla_\theta V \cdot \nabla_\theta V_s + \kappa_\sg \int_{\partial \mathbb{S}^n_+}V^pV_s\\
& = J_1 \int_{\mathbb{S}^n_+}\theta_1^{1-2\sg}(V_s)^2.
\endaligned
\end{equation}
For any $s\in (-\infty, 0)$, we define 
$$
\aligned
\widetilde{E}(s):=   &   \frac{1}{2}\int_{\mathbb{S}^n_+}\theta_1^{1-2\sg} (V_s)^2 -\frac{J_2}{2}\int_{\mathbb{S}^n_+}\theta_1^{1-2\sg} V^2 - \frac{1}{2}\int_{\mathbb{S}^n_+}\theta_1^{1-2\sg} |\nabla_\theta V|^2 \\
&   +\frac{\kappa_\sg}{p+1}\int_{\partial \mathbb{S}^n_+} V^{p+1}.
\endaligned
$$
Then, by \eqref{P30122}  we get
\begin{equation}\label{P30133}
\frac{d}{ds} \widetilde{E}(s)= J_1 \int_{\mathbb{S}^n_+}\theta_1^{1-2\sg}(V_s)^2.
\end{equation}
Note that  
$$
\begin{cases}
J_1 \geq 0~~~~~~~~~ & \textmd{when} ~~ p \leq p_S(\al),\\
J_1 < 0~~~~~~~~~ & \textmd{when} ~~ p > p_S(\al). 
\end{cases}
$$
Hence, $\widetilde{E}(s)$ is non-decreasing in $s\in (-\infty, 0)$ if $p \leq p_S(\al)$ and $\widetilde{E}(s)$ is non-increasing  in $s\in (-\infty, 0)$ if $p >  p_S(\al)$.

\vskip0.10in

Now, rescaling back to $U$,  we have
$$
\aligned
& \int_{\mathbb{S}^n_+}\theta_1^{1-2\sg} (V_s)^2  \\
&  = \int_{\mathbb{S}^n_+}\theta_1^{1-2\sg} \left(\frac{2\sg +\al}{p-1} r^{\frac{2\sg +\al}{p-1}-1}U + r^{\frac{2\sg +\al}{p-1}} U_r\right)^2r^2\\
& = r^{\frac{2(p+1)\sg +2 \al}{p-1} -n}\int_{\partial^+\mathcal{B}_r^+} t^{1-2\sg}\left(\frac{(2\sg+\al)^2}{(p-1)^2}  r^{-1}U^2 +\frac{2(2\sg+\al)}{p-1} U \frac{\partial U}{\partial \nu} +r |\frac{\partial U}{\partial \nu}|^2\right),
\endaligned
$$
$$
\int_{\mathbb{S}^n_+}\theta_1^{1-2\sg} |\nabla_\theta V|^2=r^{\frac{2(p+1)\sg + 2\al}{p-1} -n+1}\int_{\partial^+\mathcal{B}_r^+} t^{1-2\sg} \left( |\nabla U|^2 - |\frac{\partial U}{\partial \nu}|^2 \right),
$$
$$
\int_{\mathbb{S}^n_+}\theta_1^{1-2\sg} V^2=r^{\frac{2(p+1)\sg +2\al}{p-1} -n-1}\int_{\partial^+\mathcal{B}_r^+} t^{1-2\sg} U^2,
$$
$$
\int_{\partial \mathbb{S}^n_+} V^{p+1} = r^{\frac{(2\sg +\al)(p+1)}{p-1} -n+1}\int_{\partial B_r} u^{p+1}.
$$
Substituting these into \eqref{P30133} and noting  that $s=\ln r$ is increasing  in $r$, we easily obtain that $E(r;U)$ is non-decreasing in $r \in (0, 1)$ if $p \leq p_S(\al)$ and  it  is  non-increasing  in $r \in (0, 1)$  if $p >  p_S(\al)$.

\vskip0.10in

If $U$ is  a nonnegative  solution of  \eqref{Iso3E}, we just need to replace $s \in (-\infty, 0)$ in the above proof  with $s \in (0, \infty)$. The proof is finished. 
\ep

By using the monotonicity of $E(r; U)$,  we prove the following proposition, which will play an essential  role in deriving the lower bound of singular positive  solutions.

\begin{proposition}\label{P303}
Let $U$ be a nonnegative weak solution of \eqref{Iso3} with $-2 \sigma< \alpha < 2\sg $ and $ \frac{n + \al}{n-2\sg} < p < \frac{n+2\sg}{n-2\sg}$.  If
$$
\liminf_{|x| \rightarrow 0}|x|^{\frac{2\sg +\al}{p-1}} u(x)=0,
$$
then
$$
\lim_{|x| \rightarrow 0}|x|^{\frac{2\sg +\al}{p-1}} u(x)=0. 
$$
\end{proposition}

\bp
We consider separately the case $p \leq p_S(\al)$ and the case $p>p_S(\al)$. 

\vskip0.10in

{\bf  Case 1: $p \leq p_S(\al)$}.    Suppose by contradiction that
$$
\liminf_{|x| \rightarrow 0}|x|^{\frac{2\sg +\al}{p-1}} u(x)=0 ~~~~ \textmd{and} ~~~~  \limsup_{|x| \rightarrow 0}|x|^{\frac{2\sg +\al }{p-1}} u(x)=C > 0. 
$$
Then there exist two sequences of points $\{x_i\}$ and $\{y_i\}$ satisfying
$$
x_i \rightarrow 0, ~~~y_i \rightarrow 0~~~~\textmd{as}~ i \rightarrow \infty,              
$$
such that
$$
|x_i|^{\frac{2\sg +\al}{p-1}} u(x_i)\rightarrow 0~~~ \textmd{and}~~~|y_i|^{\frac{2\sg +\al}{p-1}} u(y_i)\rightarrow C>0~~~\textmd{as}~ i \rightarrow \infty.
$$
Let $g(r)=r^{\frac{2\sg +\al}{p-1}} \bar{u}(r)$, where  $\bar{u}(r)=\frac{1}{|\partial B_r|}\int_{\partial B_r} u$ denotes the spherical average of $u$ over $\partial B_r$. By the Harnack inequality \eqref{Har01}, we have
$$
\liminf_{r \rightarrow 0}g(r)=0 ~~~~ \textmd{and} ~~~~  \limsup_{r \rightarrow 0}g(r)=C > 0. 
$$
Hence, there exists a  sequence of local minimum points $r_i$ of $g(r)$ such that 
$$
\lim_{i\rightarrow \infty} r_i=0 ~~~~ \textmd{and} ~~~~ \lim_{i\rightarrow \infty} g(r_i)=0. 
$$
Define 
$$
V_i(X)=\frac{U(r_i X)}{U(r_i e_1)},
$$
where $e_1=(1,0,\cdots,0) \in \mathbb{R}^{n+1}$.   It follows  from   Harnack inequality \eqref{Har01} that $V_i$ is locally uniformly bounded away from the origin  and   satisfies   
\begin{equation}\label{P303001}
\begin{cases}
-\textmd{div}(t^{1-2\sg}  \nabla V_i)=0~~~~~~~~~& \textmd{in} ~\mathbb{R}^{n+1}_+,\\
\frac{\partial V_i}{\partial \nu^\sg}(x, 0)=\kappa_\sg \left(r_i^{\frac{2\sg +\al}{p-1}} U(r_ie_1) \right)^{p-1} |x|^\al  V_i^p(x, 0)~~~~~~~~& \textmd{on} ~\mathbb{R}^n \backslash \{0\}.
\end{cases}
\end{equation}
Note that by the Harnack inequality \eqref{Har01}, $r_i^{\frac{2\sg +\al}{p-1}} U(r_ie_1)\rightarrow 0$ as $i\rightarrow \infty$.  By Corollary 2.10 and Theorem 2.15  in \cite{J-L-X}  there exists $\gamma \in (0, 1)$ such that for every $R > 1 > r > 0$, 
$$
\|V_i\|_{W^{1,2}(t^{1-2\sg}, \mathcal{B}_R^+ \backslash \overline{\mathcal{B}}_r^+)} + \|V_i\|_{C^\gamma (\mathcal{B}_R^+ \backslash \overline{\mathcal{B}}_r^+)} + \|v_i\|_{C^{2,\gamma}(B_R \backslash B_r)} \leq C(R, r),
$$
where $v_i(x)=V_i(x,0)$ and $C(R, r)$ is independent of $i$. Then after passing to a subsequence, $\{V_i\}$ converges to a nonnegative function $V  \in W_{loc}^{1,2}(t^{1-2\sg}, \overline{\mathbb{R}^{n+1}_+} \backslash \{0\}) \cap C^\gamma_{loc}(\overline{\mathbb{R}^{n+1}_+} \backslash \{0\})$ satisfying 
\begin{equation}\label{P303002}
\begin{cases}
-\textmd{div}(t^{1-2\sg}  \nabla V)=0~~~~~~~~~& \textmd{in} ~\mathbb{R}^{n+1}_+,\\
\frac{\partial V}{\partial \nu^\sg}(x, 0)=0~~~~~~~~& \textmd{on} ~\mathbb{R}^n \backslash \{0\}.
\end{cases}
\end{equation}
By a B\^{o}cher type theorem in \cite{J-L-X}, we have
$$
V(X)=\frac{a}{|X|^{n-2\sg}} + b,
$$
where $a, b$ are nonnegative constants.  Recall that $r_i$ is a  local minimum point  of $g(r)$ for every $i$ and note that 
$$
r^{\frac{2\sg +\al}{p-1}} \bar{v}_i(r) =r^{\frac{2\sg +\al}{p-1}}\frac{1}{|\partial B_r|} \int_{\partial B_r}v_i=\frac{1}{U(r_ie_1)}r^{\frac{2\sg +\al}{p-1}}\bar{u}(r_i r)=\frac{1}{U(r_ie_1)r_i^{\frac{2\sg +\al}{p-1}}} g(r_i r). 
$$
Hence,  we have
\begin{equation}\label{P303003}
\frac{d}{dr} \left[ r^{\frac{2\sg +\al}{p-1}} \bar{v}_i(r)\right] \Bigg|_{r=1} = \frac{r_i}{U(r_ie_1)r_i^{\frac{2\sg +\al }{p-1}}} g^\prime (r_i )=0.
\end{equation}
Let $v(x)=V(x,0)$. Then   $v_i(x) \rightarrow v(x)$ in $C_{loc}^2(\mathbb{R}^n \backslash \{0\})$. By \eqref{P303003}  we obtain 
$$
\frac{d}{dr} \left[ r^{\frac{2\sg +\al}{p-1}} \bar{v}(r)\right] \Bigg|_{r=1} =0,
$$
which implies that
\begin{equation}\label{P303004}
a\left(\frac{2\sg +\al}{p-1} - (n-2\sg)\right) +  \frac{(2\sg +\al) b}{p-1}=0.
\end{equation}
On the other hand,  $V(e_1)=1$ implies 
\begin{equation}\label{P303005}
a+b=1.
\end{equation}
Combining  \eqref{P303004} with \eqref{P303005}, we get
$$
a=\frac{2\sg +\al}{(p-1)(n-2\sg)}~~~~~ \textmd{and} ~~~~~ b=1- \frac{2\sg +\al}{(p-1)(n-2\sg)}.
$$
Since $-2\sg < \al < 2\sg$ and $\frac{n +\al}{n-2\sg} < p$, we have $0 < a, b < 1$. Next we compute $E(r; U)$.

\vskip0.10in

It follows from Proposition 2.19 in \cite{J-L-X} that $|\nabla_x V_i|$ and $|t^{1-2\sg} \partial_t V_i|$ are locally uniformly bounded in $C_{loc}^\beta (\overline{\mathbb{R}^{n+1}_+}\backslash \{0\})$ for some $\beta >0$. Hence,  there exists a  constant $C>0$ such that 
$$
|\nabla_x U(X)| \leq C r_i^{-1} U(r_i e_1)= o(1) r_i^{-\frac{2\sg +\al}{p-1} - 1}~~~~~~ \textmd{for}  ~ \textmd{all} ~|X|=r_i
$$
and
$$
|t^{1-2\sg }\partial_t U(X)| \leq C r_i^{-2\sg} U(r_i e_1)= o(1) r_i^{-\frac{2\sg +\al}{p-1} - 2\sg}~~~~~~ \textmd{for}  ~ \textmd{all} ~|X|=r_i. 
$$
By the Harnack inequality \eqref{Har01}, we also have
$$
U(X) \leq C U(r_i e_1) = o(1) r_i^{-\frac{2\sg +\al}{p-1}}~~~~~~ \textmd{for}  ~ \textmd{all} ~|X|=r_i. 
$$
Thus, we   estimate 
$$
\aligned
r_i^{\frac{2(p+1)\sg +2\al}{p-1}-n+1} \int_{\partial^+\mathcal{B}_{r_i}^+} t^{1-2\sg} |\nabla U|^2 & \leq  r_i^{\frac{2(p+1)\sg +2\al}{p-1}-n+1} \bigg(o(1) r_i^{-\frac{4\sg +2\al}{p-1} -2} \int_{\partial^+\mathcal{B}_{r_i}^+} t^{1-2\sg} \\
&  ~~~~ + o(1) r_i^{-\frac{4\sg +2\al}{p-1} -4\sg} \int_{\partial^+\mathcal{B}_{r_i}^+} t^{2\sg-1} \bigg)\\
& \leq C o(1), 
\endaligned
$$
$$
r_i^{\frac{2(p+1)\sg +2\al}{p-1}-n-1} \int_{\partial^+\mathcal{B}_{r_i}^+} t^{1-2\sg} U^2 \leq  o(1) r_i^{2\sg-n-1}\int_{\partial^+\mathcal{B}_{r_i}^+} t^{1-2\sg} \leq C o(1)
$$
and
$$
r_i^{\frac{(2\sg +\al)(p+1)}{p-1}-n+1} \int_{\partial B_{r_i}} u^{p+1} \leq C o(1),
$$
where the constant $C$ is independent of $i$.  By the definition of $E(r; U)$, we have  
$$
\lim_{i \rightarrow \infty} E(r_i; U)=0.
$$
Since $E(r; U)$ is non-decreasing in $r \in (0, 1)$ for  this case, we obtain
\be\label{Dec01}
E(r; U) \geq 0~~~~~ \textmd{for} ~ \textmd{all} ~ r\in (0, 1).
\ee
On the other hand, by the scaling invariance of $E(r; U)$, we have for every  $i$ that  
$$
0 \leq E(r_i; U)=E\left(1; r_i^{\frac{2\sg +\al}{p-1}} U(r_i X) \right) = E\left(1; r_i^{\frac{2\sg +\al}{p-1}} U(r_i e_1)V_i \right).
$$
Hence,  we have
$$
\aligned
 0 \leq  & \int_{\partial^+\mathcal{B}_1^+} t^{1-2\sg} | \frac{\partial V_i}{\partial \nu} |^2 + \frac{2\sg +
 \al}{p-1} \int_{\partial^+\mathcal{B}_1^+} t^{1-2\sg}\frac{\partial V_i}{\partial \nu} V_i \\
& + \frac{2\sg +\al}{p-1}\left(\frac{2\sg+\al}{p-1} - \frac{n-2\sg}{2}\right)   \int_{\partial^+\mathcal{B}_1^+} t^{1-2\sg} V_i^2  \\
& -  \frac{1}{2} \int_{\partial^+\mathcal{B}_1^+} t^{1-2\sg} |\nabla V_i|^2  +  \frac{\kappa_\sg}{p+1}\int_{\partial B_1} \left(r_i^{\frac{2\sg +\al}{p-1}} U(r_i e_1)\right)^{p-1}V_i^{p+1}.
\endaligned
$$
Letting $i \rightarrow \infty$, we obtain
$$
\aligned
 0  & \leq  \int_{\partial^+\mathcal{B}_1^+} t^{1-2\sg} | \frac{\partial V}{\partial \nu} |^2 + \frac{2\sg +\al}{p-1} \int_{\partial^+\mathcal{B}_1^+} t^{1-2\sg}\frac{\partial V}{\partial \nu} V \\
& ~~~~ + \frac{2\sg +\al}{p-1}\left(\frac{2\sg+\al}{p-1} - \frac{n-2\sg}{2}\right)   \int_{\partial^+\mathcal{B}_1^+} t^{1-2\sg} V^2 -  \frac{1}{2} \int_{\partial^+\mathcal{B}_1^+} t^{1-2\sg} |\nabla V|^2 \\
& =a^2(n-2\sg)^2 \int_{\partial^+\mathcal{B}_1^+} t^{1-2\sg} - a(n-2\sg)\frac{2\sg +\al}{p-1}\int_{\partial^+\mathcal{B}_1^+} t^{1-2\sg} \\
& ~~~~ +\frac{2\sg +\al}{p-1}\left(\frac{2\sg +\al}{p-1} - \frac{n-2\sg}{2}\right)   \int_{\partial^+\mathcal{B}_1^+} t^{1-2\sg}-\frac{1}{2}a^2(n-2\sg)^2\int_{\partial^+\mathcal{B}_1^+} t^{1-2\sg}\\
& = \frac{1}{2}\frac{2\sg +\al}{p-1} \left(\frac{2\sg +\al}{p-1} - (n-2\sg)\right)   \int_{\partial^+\mathcal{B}_1^+} t^{1-2\sg} <0.
\endaligned
$$
Here we have used the facts  $2\sg +\al >0$ and   $\frac{2\sg +\al}{p-1} - (n-2\sg) <0$  in the last inequality.  We get a contradiction.  This completes the proof of Case 1.  

\vskip0.15in

{\bf  Case 2: $p > p_S(\al)$}. In this case,  it follows from  Proposition \ref{P302} (2)   that  $E(r; U)$ is non-increasing in $r \in (0, 1)$.   If we proceed as in the proof of Case 1, then we obtain $E(r; U) \leq 0$ for $r\in (0, 1)$ in \eqref{Dec01}, and so  we cannot  get a contradiction in the final proof.  Thus,   a new method is needed to deal with this supercritical case.    In fact, the following method is available for all $p \neq p_S(\al)$.      

\vskip0.10in
{\it Step 1. If $\liminf_{|x| \to 0} |x|^{\frac{2\sg +\al}{p-1}} u(x) =0$, then 
\be\label{LEU0}
\lim_{r \to 0^+} E(r; U)=0. 
\ee}
Since $\liminf_{|x| \to 0} |x|^{\frac{2\sg +\al}{p-1}} u(x) =0$,  there exists a sequence of points  $\{x_i\}$ such that
$$
x_i \to 0~~~~~~ \textmd{and} ~~~~~~ |x_i|^{\frac{2\sg +\al}{p-1}} u(x_i) \to 0~~~~~~ \textmd{as} ~ i \to \infty. 
$$
Let $r_i:=|x_i|$. By the Harnack inequality \eqref{Har01}, 
$$
r_i^{\frac{2\sg +\al}{p-1}} U(r_ie_1)\rightarrow 0 ~~~~~~ \textmd{as} ~ i\rightarrow \infty, 
$$
where $e_1=(1,0,\cdots,0) \in \mathbb{R}^{n+1}$. Define
$$
W_i(X)=r_i^{\frac{2\sg +\al}{p-1}} U(r_i X), ~~~~~~~ X \in \mathcal{B}_{1/r_i}^+ \backslash \{0\}. 
$$
It follows  from Proposition \ref{SDE} and  Harnack inequality \eqref{Har01} that $W_i$ is locally uniformly bounded away from the origin. Moreover, $W_i$ satisfies 
\be\label{Wi01}
\begin{cases}
-\textmd{div}(t^{1-2\sg}  \nabla W_i)=0~~~~~~~~~& \textmd{in} ~ \mathcal{B}_{1/r_i}^+,\\
\frac{\partial W_i}{\partial \nu^\sg}(x, 0)=\kappa_\sg |x|^\al W_i^p(x, 0)~~~~~~~~& \textmd{on} ~ \partial^0\mathcal{B}_{1/r_i}^+ \backslash  \{0\}, 
\end{cases}
\ee
and
\be\label{Wi02}
W_i(e_1) \rightarrow 0 ~~~~~~ \textmd{as} ~ i\rightarrow \infty.
\ee
By Corollary 2.10, Theorem 2.15 and Proposition 2.19 in \cite{J-L-X}  there exists $\gamma \in (0, 1)$ such that for every $R > 1 > r > 0$
$$
\|W_i\|_{W^{1,2}(t^{1-2\sg}, \mathcal{B}_R^+ \backslash \overline{\mathcal{B}}_r^+)} + \|W_i\|_{C^\gamma (\mathcal{B}_R^+ \backslash \overline{\mathcal{B}}_r^+)}   + \| t^{1-2\sg} \partial_t W_i \|_{C^\gamma (\mathcal{B}_R^+ \backslash \overline{\mathcal{B}}_r^+)}  \leq C(R, r),
$$
where  $C(R, r)$ is independent of $i$. Then after passing to a subsequence, $\{W_i\}$ converges to a nonnegative function $W  \in W_{loc}^{1,2}(t^{1-2\sg}, \overline{\mathbb{R}^{n+1}_+} \backslash \{0\}) \cap C^\gamma_{loc}(\overline{\mathbb{R}^{n+1}_+} \backslash \{0\})$ satisfying 
\begin{equation}\label{Whole}
\begin{cases}
-\textmd{div}(t^{1-2\sg}  \nabla W)=0~~~~~~~~~& \textmd{in} ~\mathbb{R}^{n+1}_+,\\
\frac{\partial W}{\partial \nu^\sg}(x, 0)=\kappa_\sg |x|^\al W^p(x, 0) ~~~~~~~~& \textmd{on} ~\mathbb{R}^n \backslash \{0\}.
\end{cases}
\end{equation}
By \eqref{Wi02}  we have $W(e_1)=0$. This together with  Lemma \ref{HI01} implies  that $W \equiv 0$ in $\overline{\mathbb{R}^{n+1}_+} \backslash \{0\}$. 
Since  $E(r; U)$ is invariant under the scaling,
$$
\lim_{i \to \infty} E(r_i; U)=\lim_{i \to \infty} E(1; W_i) =  E(1; W)=0.
$$
By the monotonicity of $E(r; U)$  (Proposition \ref{P302}), we obtain
$$
\lim_{r \to 0^+} E(r; U)=0.
$$

{\it Step 2. Let $W$ be a nonnegative solution of \eqref{Whole} in $\mathbb{R}^{n+1}_+$.  If $E(r; W) \equiv 0$  for $r \in (0, \infty)$, then
$$
W \equiv 0 ~~~~~~~~  \textmd{in} ~ \overline{\mathbb{R}^{n+1}_+} \backslash \{0\}. 
$$ }
Since $p \neq p_S(\al)$,  we have $J_1=\frac{n-2\sg}{p-1}\left( \frac{n+2\sg +2\al}{n-2\sg} - p \right) \neq 0$.  By Proposition \ref{P302}  we get  
$$
\frac{\partial W}{\partial r} + \frac{2\sg +\al}{p-1} \frac{W}{r} =0 ~~~~~~ \textmd{in}  ~ \mathbb{R}^{n+1}_+. 
$$
This implies that  $W$ is homogeneous of degree $-\frac{2\sg +\al}{p-1}$. That is, there exists $\varphi \in C^2(\mathbb{S}^{n}_+)$ such that
$$
W(X)= r^{-\frac{2\sg +\al}{p-1}} \varphi(\theta), 
$$
where $X=(x ,t) =r \theta$ with $r=|X|$ and $\theta=\frac{X}{|X|}$. Let $\theta_1=\frac{t}{|X|}$ denote  the component of $\theta$ in the $t$ direction. A calculation  similar to the proof of Proposition \ref{P302} shows that  $\varphi$ satisfies 
\begin{equation}\label{Hom01}
\begin{cases}
- \theta_1^{2\sg-1} \text{div}_\theta (\theta_1^{1-2\sg} \nabla_\theta \varphi) + J_2 \varphi =0~~~~~~~~~& \textmd{on} ~\mathbb{S}^n_+ ,\\
-\lim_{\theta_1\rightarrow 0^+} \theta_1^{1-2\sg} \partial_{\theta_1} \varphi = \kappa_\sg \varphi^p~~~~~~~~& \textmd{on} ~ \partial \mathbb{S}^n_+,\\
\end{cases}
\end{equation}
where  
$$
J_2=\frac{2\sg +\al}{p-1}\left(n-2\sg-\frac{2\sg + \al}{p-1}\right).
$$
Multiplying \eqref{Hom01} by $\varphi$ and integrating on $\mathbb{S}^n_+$, we obtain 
\be\label{Hom02}
\int_{\mathbb{S}^n_+} \theta_1^{1-2\sg} |\nabla_\theta \varphi|^2 + J_2 \int_{\mathbb{S}^n_+} \theta_1^{1-2\sg} \varphi^2 =\kappa_\sg \int_{\partial \mathbb{S}^n_+} \varphi^{p+1}. 
\ee
On the other hand,  by  the proof of Proposition \ref{P302}, $E(r; W) \equiv 0$ gives 
\be\label{Hom03}
 -\frac{J_2}{2}\int_{\mathbb{S}^n_+}\theta_1^{1-2\sg} \varphi^2 - \frac{1}{2}\int_{\mathbb{S}^n_+}\theta_1^{1-2\sg} |\nabla_\theta \varphi|^2 
+\frac{\kappa_\sg}{p+1} \int_{\partial \mathbb{S}^n_+} \varphi^{p+1} =0.
\ee
Combining  \eqref{Hom02} with  \eqref{Hom03}, we easily get 
$$
\left( 1 - \frac{2}{p+1}\right) \int_{\partial \mathbb{S}^n_+} \varphi^{p+1} =0, 
$$
and so $\varphi \equiv  0$ on $\partial \mathbb{S}^n_+$. By \eqref{Hom02} and $J_2 >0$, we obtain $\varphi =  0$ on  $\mathbb{S}^n_+$.    Hence $W \equiv 0$ in $\mathbb{R}^{n+1}_+$.  

\vskip0.10in

{\it Step 3. End of Proof.} For $\lambda >0$ small, define 
$$
U^\lambda (X)=\lambda^{\frac{2\sg +\al}{p-1}} U(\lambda X). 
$$
Then $U^\lambda$ is also a nonnegative solution of \eqref{Iso3} in $\mathcal{B}_{1/\lambda}^+$.  It follows  from Proposition \ref{SDE} and Harnack inequality \eqref{Har01} that $U^\lambda$ is locally uniformly bounded away from the origin.  By Corollary 2.10, Theorem 2.15 and Proposition 2.19 in \cite{J-L-X}  there exists $\gamma \in (0, 1)$ such that for every $R > 1 > r > 0$
$$
\|U^\lambda \|_{W^{1,2}(t^{1-2\sg}, \mathcal{B}_R^+ \backslash \overline{\mathcal{B}}_r^+)} + \|U^\lambda \|_{C^\gamma (\mathcal{B}_R^+ \backslash \overline{\mathcal{B}}_r^+)}   + \| t^{1-2\sg} \partial_t U^\lambda  \|_{C^\gamma (\mathcal{B}_R^+ \backslash \overline{\mathcal{B}}_r^+)}  \leq C(R, r),
$$
where  $C(R, r)$ is independent of $\lambda$. Hence, there is a subsequence $\lambda_i$ of $\lambda \to 0$   such that $\{U^{\lambda_i}\}$ converges to a nonnegative function $U^0  \in W_{loc}^{1,2}(t^{1-2\sg}, \overline{\mathbb{R}^{n+1}_+} \backslash \{0\}) \cap C^\gamma_{loc}(\overline{\mathbb{R}^{n+1}_+} \backslash \{0\})$ satisfying 
$$
\begin{cases}
-\textmd{div}(t^{1-2\sg}  \nabla U^0)=0~~~~~~~~~& \textmd{in} ~\mathbb{R}^{n+1}_+,\\
\frac{\partial U^0}{\partial \nu^\sg}(x, 0)=\kappa_\sg |x|^\al ( U^0(x, 0) )^p  ~~~~~~~~& \textmd{on} ~\mathbb{R}^n \backslash \{0\}.
\end{cases}
$$
Moreover, by the scaling invairance of $E$ and Step 1,  we have for any $r >0$ that 
$$
E(r; U^0)= \lim_{i \to \infty} E(r; U^{\lambda_i}) =  \lim_{i \to \infty} E(\lambda_i r; U) =\lim_{r \to 0^+}E(r; U)=0. 
$$
The conclusion of  Step 2  gives  $U^0 \equiv 0$ in $\overline{\mathbb{R}^{n+1}_+} \backslash \{0\}$.  Since the limiting function $U^0$ is unique for any subsequence of $\lambda \to 0$, we obtain 
$$
\lim_{\lambda \to 0} U^\lambda = 0 ~~~~~~~~ \textmd{in}~~~ C_{loc}^{\gamma/2}(\overline{\mathbb{R}^{n+1}_+} \backslash \{0\}).  
$$
In particular,  
$$
\lim_{\lambda \to 0} \lambda^{\frac{2\sg +\al}{p-1}} u(\lambda x) = 0 ~~~~~~~~ \textmd{uniformly}~ \textmd{for}~ x\in \partial B_1,  
$$
which  immediately implies  $\lim_{|x|\to 0} |x|^{\frac{2\sg +\al}{p-1}}u(x) =0$.   
\ep

\begin{proposition}\label{Re089}
Let $U$ be a nonnegative weak solution of \eqref{Iso3} with $-2 \sigma< \alpha < 2\sg $ and $ \frac{n + \al}{n-2\sg} < p < \frac{n+2\sg}{n-2\sg}$.  If
$$
\lim_{|x| \rightarrow 0}|x|^{\frac{2\sg +\al}{p-1}} u(x)=0,
$$
then the singularity at $x=0$ is removable, i.e., $U(x,t)$ can be extended to a continuous function near the origin $0$. 
\end{proposition}

\bp
By the Harnack inequality \eqref{Har01},  we have 
\be\label{ULim12}
\lim_{|X|\to 0} |X|^{\frac{2\sg +\al}{p-1}} U(X) =0. 
\ee
For any $0 < \mu < n-2\sg$ and $0 < \delta < \frac{1}{2}$,  as in \cite{C-J-S-X},  we define 
$$
\Psi_\mu(X) := |X|^{-\mu} \left( 1 -\delta \left( \frac{t}{|X|}  \right)^{2\sg}  \right),~~~~~ 
$$
where $X =(x, t) \neq 0$. Then $\Psi_\mu$ satisfies  
$$
\begin{cases}
 -\textmd{div} (t^{1-2\sg} \nabla \Psi_\mu(X))= t^{1-2\sg} |X|^{-(\mu +2)} \left( \mu(n -2\sg -\mu) - \frac{\delta(\mu + 2\sg)(n-\mu) t^{2\sg}}{|X|^{2\sg}} \right),\\
-\lim_{t \to 0} t^{1-2\sg} \partial_t \Psi_\mu (x, t) = 2\sg \delta |x|^{- 2\sg} \Psi_\mu(x, 0),~~~~~~ x \neq 0. 
\end{cases}
$$
Let $\mu_0=\frac{2\sg +\al}{p-1}$ and $\tau \in (0, \frac{2\sg +\al}{p-1})$ be fixed. Note that $ 0 < \frac{2\sg +\al}{p-1} < n-2\sg$ due to $-2\sg < \al$ and $\frac{n+\al}{n-2\sg} <p$. Let
$$
\Psi = \epsilon \Psi_{\mu_0} + C \Psi_\tau, 
$$
where $\epsilon, C$ are positive constants. Then we can choose small  $\delta =\delta(\tau, \al, \sg, p, n) \in (0, \frac{1}{2})$  such that
$$
\begin{cases}
 -\textmd{div} (t^{1-2\sg} \nabla \Psi ) \geq 0 ~~~~~~~  & \textmd{in} ~ \mathcal{B}_1^+,\\
\frac{\partial \Psi}{\partial \nu^\sg} (x, 0) = 2\sg \delta |x|^{- 2\sg} \Psi(x, 0) ~~~~~~~~& \textmd{on} ~ \partial^0\mathcal{B}_1^+ \backslash  \{0\}. 
\end{cases}
$$
Let $a(x):=\kappa_\sg |x|^\al u^{p-1}(x)$. By the assumption we have $\lim_{|x|\to 0} a(x)|x|^{2\sg} =0$.  Hence, there exists $r_0 \in (0, 1)$ such that
$$
a(x) \leq 2\sg \delta |x|^{-2\sg} ~~~~~~~~ \textmd{for} ~ 0< |x| \leq r_0. 
$$
Thus, we have  
$$
\begin{cases}
 -\textmd{div} (t^{1-2\sg} \nabla ( \Psi - U) ) \geq 0 ~~~~~~~  & \textmd{in} ~ \mathcal{B}_{r_0}^+,\\
\frac{\partial ( \Psi -U )}{\partial \nu^\sg} (x, 0) \geq  2\sg \delta |x|^{- 2\sg} ( \Psi(x, 0)- U(x, 0))  ~~~~~~~~& \textmd{on} ~ \partial^0\mathcal{B}_{r_0}^+ \backslash  \{0\}. 
\end{cases}
$$
Furthermore, we note that
$$
\Psi(X) \geq \frac{\epsilon}{2} |X|^{-\frac{2\sg +\al}{p-1}}~~~~~~~ \textmd{for}  ~ X \in \mathcal{B}_1^+ \backslash \{0\}. 
$$
Hence,  for any $\epsilon > 0$, by \eqref{ULim12} there exists $r_\epsilon > 0$ small such that
$$
\Psi \geq U ~~~~~~~~~~ \textmd{in}  ~~  \overline{\mathcal{B}_{r_\epsilon}^+} \backslash \{0\}.
$$
On the other hand, we can choose  $C=C(\tau, r_0, U) $ sufficiently large so  that 
$$
\Psi \geq U ~~~~~~~~~ \textmd{on}  ~~  \partial^+\mathcal{B}_{r_0}^+. 
$$
The maximum principle gives that
$$
\Psi \geq U ~~~~~~~~~~ \textmd{in}  ~~  \overline{\mathcal{B}_{r_0}^+} \backslash \{0\}.
$$
Letting  $\epsilon \to 0$, we have 
\be\label{Int8910}
U(X) \leq C(\tau, r_0, U) \Psi_\tau (X) \leq C(\tau, r_0, U) |X|^{-\tau}  ~~~~~~~~~~ \textmd{in}  ~~  \overline{\mathcal{B}_{r_0}^+} \backslash \{0\}.
\ee
By standard rescaling arguments and Proposition 2.19 in \cite{J-L-X}, we obtain 
\be\label{Int8920}
|\nabla_x U(X)| \leq C(\tau, r_0, U)  |X|^{-\tau -1}  ~~~~~~~~~~ \textmd{in}  ~~ \mathcal{B}_{r_0/2}^+  \backslash \{0\}
\ee
and 
\be\label{Int8930}
|t^{1-2\sg }\partial_t U(X)| \leq C(\tau, r_0, U)  |X|^{-\tau -2\sg}  ~~~~~~~~~~ \textmd{in}  ~~ \mathcal{B}_{r_0/2}^+  \backslash \{0\}. 
\ee
Since $\tau \in (0, \frac{2\sg +\al}{p-1})$ is arbitrary, it is not difficult to verify that $U \in W^{1, 2}(t^{1-2\sg} ,\mathcal{B}_1^+)$. Next we will prove that  $U$ is a nonnegative weak solution of 
\be\label{W-weak}
\begin{cases}
-\textmd{div}(t^{1-2\sg}  \nabla U)=0~~~~~~~~~& \textmd{in} ~ \mathcal{B}_1^+,\\
\frac{\partial U}{\partial \nu^\sg}(x, 0)=\kappa_\sg |x|^\al U^p(x, 0)~~~~~~~~& \textmd{on} ~ \partial^0\mathcal{B}_1^+.
\end{cases}
\ee
In fact, for $\epsilon >0$ small, let $\eta_\epsilon \in C^\infty(\mathbb{R}^{n+1})$ be a cut-off  function satisfying
$$\eta_\epsilon (X) =
\begin{cases}
0~~~~~~~~~~ \textmd{for} ~ |X|\leq \epsilon,\\ 
1~~~~~~~~~~ \textmd{for} ~ |X|\geq  2\epsilon 
\end{cases}
$$
and
$$
|\nabla  \eta_\epsilon(X)| \leq C \epsilon^{-1}.
$$
For any $\psi \in C_c^\infty\left((\mathcal{B}_1^+  \cup \partial^0\mathcal{B}_1^+) \right)$, using $\psi\eta_\epsilon$ as a test function in \eqref{Solu} gives 
\be\label{Test069}
\int_{\mathcal{B}_1^+} t^{1-2\sg} \nabla U \cdot\nabla (\psi\eta_\epsilon)= \kappa_\sg \int_{B_1} |x|^\al u^p(x) \psi\eta_\epsilon. 
\ee
But
$$
\aligned
\left| \int_{\mathcal{B}_1^+} t^{1-2\sg} \psi \nabla U \cdot\nabla \eta_\epsilon  \right| & \leq C \epsilon^{-1} \left( \int_{\mathcal{B}_{2\epsilon}^+ \backslash \mathcal{B}_\epsilon^+}  t^{1-2\sg} |\nabla U|^2 \right)^{1/2} \left( \int_{\mathcal{B}_{2\epsilon}^+ \backslash \mathcal{B}_\epsilon^+}  t^{1-2\sg}\right)^{1/2}\\
& \leq C \epsilon^{\frac{n- 2\sg}{2}}  \left( \int_{\mathcal{B}_{2\epsilon}^+ \backslash \mathcal{B}_\epsilon^+}  t^{1-2\sg} |\nabla U|^2 \right)^{1/2} \to 0~~~~~~~ \textmd{as} ~ \epsilon \to 0. 
\endaligned
$$
By  \eqref{Int8910}  and $\al > -2\sg$,  we have $|\cdot|^\al u^p \in L_{loc}^1(B_1)$. Letting $\epsilon \to 0$ in \eqref{Test069}, we get 
$$
\int_{\mathcal{B}_1^+} t^{1-2\sg} \nabla U \cdot\nabla \psi= \kappa_\sg \int_{B_1} |x|^\al u^p(x) \psi.  
$$
Hence  $U$ is a nonnegative weak solution of \eqref{W-weak}.  Again,  by \eqref{Int8910} and $\al >-2\sg$, we obtain  
$$
|\cdot|^\al u^{p-1}  \in L^{q}(B_{1/2}) 
$$
for some $q> \frac{n}{2\sg}$. It follows from Proposition 2.6 in \cite{J-L-X} that $U$ is H\"older continuous in $\overline{\mathcal{B}_{1/2}^+}$. 
\ep

\vskip0.10in

\noindent{\bf Proof of Theorem \ref{C-T2-2020}.} The proof of Theorem \ref{C-T2-2020} is  now  just a combination of  Harnack inequality in Lemma \ref{LM-02},   Propositions \ref{SDE}, \ref{P303} and \ref{Re089}.   
\hfill$\square$

\vskip0.10in

\noindent{\bf Proof of Theorem \ref{C-T2}.} It follows from the extension theorem of Caffarelli-Silvestre \cite{C-S} and Theorem \ref{C-T2-2020}.
\hfill$\square$

\section{Precise Asymptotic Behavior} 
In this section,  we  prove Theorem \ref{AB-T1-2020},   Theorem \ref{AB-T1} and Theorem \ref{unique-2020}. We begin by showing the boundedness of the energy integral  $E(r; U)$ defined in  \eqref{EI000}.      

\begin{proposition}\label{4Le01}

Let $n \geq 2$,  $-2\sg < \al < 2\sg$ and $\frac{n+\al}{n-2\sg} < p < \frac{n+2\sg}{n-2\sg}$. 
Assume  that $U$ is a nonnegative weak solution of \eqref{Iso3} ({\it resp. } of \eqref{Iso3E}).  Then  $E(r; U)$ is uniformly bounded  in $r \in (0, \frac{1}{8})$ ({\it resp.} in $r \in (8, \infty)$).  Further, the limit 
$$
\lim_{r \to 0^+} E(r; U) ~~ ({\it resp. \lim_{r \to +\infty} E(r; U)})
$$
exists and it is finite. 
\end{proposition} 

\begin{proof}
Suppose  $U$ is a nonnegative weak solution of \eqref{Iso3}.   For any $r\in (0, \frac{1}{8})$,  define  
$$
V(X)=r^{\frac{2\sg +\al}{p-1}}U(rX), ~~~~~~~ \frac{1}{2} \leq   |X| \leq  2. 
$$
 Then $V$ satisfies
$$
\begin{cases}
-\textmd{div}(t^{1-2\sg}  \nabla V)=0~~~~~~~~~& \textmd{in} ~\mathcal{B}_2^+ \backslash \overline{\mathcal{B}_{1/2}^+},\\
\frac{\partial V}{\partial \nu^\sg}(x, 0)= \kappa_\sg |x|^{\al} v^p(x)~~~~~~~~& \textmd{on} ~B_2 \backslash \overline{B}_{1/2},\\
\end{cases}
$$
where $v(x)=V(x, 0)$. It follows from Proposition \ref{SDE} and Lemma \ref{LM-02} that 
$$
|V(X)| \leq C ~~~~~~~~ \textmd{for} ~ \textmd{all} ~ \frac{1}{2} \leq |X| \leq 2,
$$
where $C$ is a positive constant depending only on $n, p, \sg$ and $\al$.   By Proposition 2.19 in \cite{J-L-X}, we have
$$
\sup_{\frac{3}{4} \leq |X| \leq \frac{3}{2}} |\nabla_x V| + \sup_{\frac{3}{4} \leq |X| \leq \frac{3}{2}} | t^{1-2\sg}\partial_t V| \leq C. 
$$
Hence, there exists $C>0$ depending only on $n, p, \sg$ and $\al$  such that
$$
|\nabla_x U (X)|\leq C |X|^{-\frac{2\sg}{p-1} -1}~~~~~~~~ \textmd{in} ~ \mathcal{B}_{1/8}^+ \backslash \{0\}
$$
and 
$$
|t^{1-2\sg}\partial_t U(X)|\leq C |X|^{-\frac{2\sg}{p-1} -2\sg}~~~~~~~~ \textmd{in} ~ \mathcal{B}_{1/8}^+ \backslash \{0\}. 
$$
Thus, a direct computation gives 
$$
r^{\frac{2(p+1)\sg + 2\al}{p-1}-n+1} \int_{\partial^+ \mathcal{B}_r^+} t^{1-2\sg} |\nabla U| \leq  C, 
$$
$$
r^{\frac{2(p+1)\sg +2\al}{p-1}-n-1} \int_{\partial^+ \mathcal{B}_r^+} t^{1-2\sg} U^2 \leq C, 
$$
$$
r^{\frac{(p+1)(2\sg +\al)}{p-1}-n+1}\int_{\partial B_r} u^{p+1} \leq C,
$$
where  $C$ is a positive constant  depending only on $n, p, \sg$ and  $\al$. Now  we easily conclude that $E(r; U)$ is uniformly bounded in $r \in (0, \frac{1}{8})$. By the monotonicity of $E(r; U)$, we obtain that the limit 
$$\lim_{r\rightarrow 0^+} E(r; U)$$
exists and is finite.  

\vskip0.10in

Similarly, let  $U$ be a nonnegative weak solution of \eqref{Iso3E}, we can prove that $E(r; U)$ is uniformly bounded in $r\in  (8, \infty)$,  and then the limit  
$\lim_{r\rightarrow +\infty} E(r; U)$
exists and  is finite.   
\end{proof}

Next,  we show an uniqueness result  for  a degenerate elliptic equation on  $\mathbb{S}^{n}_+$. 
\begin{proposition}\label{4Le02}
Assume $n\geq 2$, $\sg \in (0, 1)$, $-2\sg < \al < 2\sg$ and $p > \frac{n+\al}{n-2\sg}$. Let $\varphi \in  C^2(\mathbb{S}^{n}_+) \cap C(\overline{\mathbb{S}^{n}_+}) $ be a solution of 
\begin{equation}\label{Hom01-01-999} 
\begin{cases}
- \theta_1^{2\sg-1} \text{div}_\theta (\theta_1^{1-2\sg} \nabla_\theta \varphi) + J_2 \varphi =0~~~~~~~~~& \textmd{on} ~\mathbb{S}^n_+ ,\\
-\lim_{\theta_1\rightarrow 0^+} \theta_1^{1-2\sg} \partial_{\theta_1} \varphi = \kappa_\sg (\varphi)^p~~~~~~~~& \textmd{on} ~ \partial \mathbb{S}^n_+,\\
\end{cases}
\end{equation}
where 
$$
J_2=\frac{2\sg +\al}{p-1}\left(n-2\sg-\frac{2\sg + \al}{p-1}\right).
$$
If $\varphi \equiv c$  on $\partial \mathbb{S}^n_+$ for some positive constant $c$,  then necessarily $c=C_{p,\sg,\al}$ and 
$$
\varphi \equiv  \omega_\alpha ~~~~~~~ \textmd{on} ~\mathbb{S}^n_+, 
$$
where $C_{p,\sg,\al}$ is given by \eqref{A},    $\omega_\alpha \in C^2(\mathbb{S}^{n}_+) \cap C(\overline{\mathbb{S}^{n}_+}) $ is the restriction of $V_\alpha$ to $\overline{\mathbb{S}^n_+}$,  and $V_\alpha$ is the Caffarelli-Silvestre extension of the function $v_\alpha=C_{p,\sg,\al}|x|^{-\frac{2\sg +\al}{p-1}}$, as defined  in \eqref{Extension}.     
\end{proposition} 
\begin{proof}
For $x \in \mathbb{R}^n \backslash \{0\} $, let $v_\alpha(x)=C_{p,\sg,\al} |x|^{-\frac{2\sg +\al}{p-1}}$ where $C_{p,\sg,\al}$ is given by \eqref{A}.  Then by Lemma 3.1 in Fall \cite{Fall}, we know  that 
$$
(-\Delta)^\sg v_\alpha(x) =  |x|^{\al} (v_\alpha)^p(x)~~~~~~~~~ \textmd{in} ~ \mathbb{R}^n \backslash \{0\},
$$  
Let $V_\alpha$ be the Caffarelli-Silvestre extension of $v_\alpha$, that is,
$$
V_\alpha(x, t) =\int_{\mathbb{R}^n} P_\sigma(x-y, t) v_\alpha(y) dy ~~~~~~ \textmd{for} ~(x, t) \in \mathbb{R}^{n+1}_+. 
$$
Then we have 
$$
\begin{cases}
-\textmd{div}(t^{1-2\sg}  \nabla V_\alpha)=0~~~~~~~~~& \textmd{in} ~ \mathbb{R}^{n+1}_+,\\
\frac{\partial V_\alpha}{\partial \nu^\sg}(x, 0)=\kappa_\sg |x|^\al  (v_\alpha)^p(x)~~~~~~~~& \textmd{on} ~ \partial^0\mathbb{R}^{n+1}_+  \backslash  \{0\}. 
\end{cases}
$$
It is easy to check that $V_\alpha$ is a homogeneous function.  Setting  $ \omega_\alpha = V_\alpha \Big|_{ \overline{\mathbb{S}^n_+} }$.  
Then, $\omega_\alpha \equiv C_{p,\sigma,\alpha}$ on $\partial \mathbb{S}^n_+$  and for $X=(x, t)=r\theta \in \mathbb{R}^{n+1}_+ \backslash \{0\}$,
$$
V_\alpha(X) =r^{-\frac{2\sg +\al}{p-1}} \omega_\alpha(\theta). 
$$
A  direct calculation shows that $\omega_\alpha$ also satisfies \eqref{Hom01-01-999}. 
Define $\omega= \varphi - \left( \frac{c}{C_{p,\sg,\al}} \right)^p  \omega_\alpha$.  Then $\omega$ satisfies 
\begin{equation}\label{00-88-26}
\begin{cases}
- \theta_1^{2\sg-1} \text{div}_\theta (\theta_1^{1-2\sg} \nabla_\theta \omega) + J_2 \omega =0~~~~~~~~~& \textmd{on} ~\mathbb{S}^n_+ ,\\
-\lim_{\theta_1\rightarrow 0^+} \theta_1^{1-2\sg} \partial_{\theta_1} \omega = 0~~~~~~~~& \textmd{on} ~ \partial \mathbb{S}^n_+. \\
\end{cases}
\end{equation}
Multiplying \eqref{00-88-26} by $\omega$ and integrating, we obtain
$$
\int_{\mathbb{S}^n_+} \theta_1^{1-2\sigma} |\nabla \omega|^2 + J_2 \int_{\mathbb{S}^n_+} \theta_1^{1-2\sigma} \omega^2 =0. 
$$
Note that $J_2>0$ because of $\alpha > -2\sigma$ and $p>\frac{n+\alpha}{n-2\sigma}$.  Thus, the above equality  leads to  $\omega \equiv0$ on $\mathbb{S}^n_+$,  and hence $\omega \equiv0$ on $\overline{\mathbb{S}^n_+}$. The proposition  follows immediately.  
\end{proof}

\vskip0.10in

\noindent{\bf Proof of Theorem \ref{AB-T1-2020}.} 
Suppose that $U$ is  a nonnegative weak solution of \eqref{Iso3-2020} and the origin $0$  is a non-removable singularity, we only need to  establish  \eqref{Beh-0D-2020}.  We consider separately the subcritical case $p < p_S(\al)$ and the supercritical case $p>p_S(\al)$.   

\vskip0.10in

{\bf Case 1:  $p < p_S(\al)$}. 
We define the scaling 
$$
U^\lambda(X)=\lambda^{\frac{2\sg +\al}{p-1}} U(\lambda X). 
$$
Then $U^\lambda$ satisfies
$$
\begin{cases}
-\textmd{div}(t^{1-2\sg}  \nabla U^\lambda)=0~~~~~~~~~& \textmd{in} ~ \mathcal{B}_{1/\lambda}^+,\\
\frac{\partial U^\lambda}{\partial \nu^\sg}(x, 0)= \kappa_\sg |x|^{\al} (U^\lambda(x, 0))^p~~~~~~~~& \textmd{on} ~ \partial^0\mathcal{B}_{1/\lambda}^+ \backslash  \{0\}. 
\end{cases}
$$
Since $0$ is a non-removable singularity,  by Theorem \ref{C-T2-2020}  there exist $C_1, C_2 >0$ such that 
\begin{equation}\label{TA01-01}
C_1 |X|^{-\frac{2\sg +\al}{p-1}}\leq U^\lambda (X) \leq C_2 |X|^{-\frac{2\sg +\al}{p-1}}~~~~~\textmd{in} ~ \overline{\mathcal{B}_{1/(2\lambda)}^+}  \backslash  \{0\}. 
\end{equation}
Thus, $U^\lambda$ is locally uniformly bounded away from the origin. It follows from Corollary 2.10 and Theorem 2.15  in \cite{J-L-X}  that  there exists $\gamma >0$ such that for every $R > 1 > r > 0$ 
$$
\|U^\lambda \|_{W^{1,2}(t^{1-2\sg}, \mathcal{B}_R^+ \backslash \overline{\mathcal{B}}_r^+)} + \|U^\lambda\|_{C^\gamma (\mathcal{B}_R^+ \backslash \overline{\mathcal{B}}_r^+)} + \|u^\lambda\|_{C^{2,\gamma}(B_R \backslash B_r)} \leq C(R, r),
$$
where $u^\lambda(x)=U^\lambda(x, 0)$ and $C(R, r)$ is independent of $\lambda$.  Then there is a subsequence $\lambda_k$ of $\lambda \rightarrow 0$  such that $\{U^{\lambda_k}\}$ converges to a nonnegative function $U^0  \in W_{loc}^{1,2}(t^{1-2\sg}, \overline{\mathbb{R}^{n+1}_+} \backslash \{0\}) \cap C^\gamma_{loc}(\overline{\mathbb{R}^{n+1}_+} \backslash \{0\})$ satisfying 
$$
\begin{cases}
-\textmd{div}(t^{1-2\sg}  \nabla U^0)=0~~~~~~~~~& \textmd{in} ~\mathbb{R}^{n+1}_+,\\
\frac{\partial U^0}{\partial \nu^\sg}(x, 0)= \kappa_\sg |x|^\al (U^0(x,0))^p~~~~~~~~& \textmd{on} ~\mathbb{R}^n \backslash \{0\}. 
\end{cases}
$$
By \eqref{TA01-01}  we have 
\begin{equation}\label{TA02-02}
C_1 |X|^{-\frac{2\sg +\al}{p-1}} \leq U^0(X) \leq C_2 |X|^{-\frac{2\sg +\al}{p-1}}~~~~~~\textmd{in} ~\overline{\mathbb{R}^{n+1}_+} \backslash \{0\}. 
\end{equation}
Moreover, by the scaling invariance of  $E(r; U)$ and Proposition  \ref{4Le01}, we have for any $r>0$ that 
$$
E(r; U^0)= \lim_{k\rightarrow \infty}E(r; U^{\lambda_k})  = \lim_{k\rightarrow \infty}E(r\lambda_k; U)=E(0^+; U).
$$
That is, $E(r; U^0)$ is a constant. It follows from  Proposition \ref{P302} that $U^0$ is homogeneous of degree $-\frac{2\sg +\al}{p-1}$. Hence, there exists $\varphi^0  \in C^2(\mathbb{S}^{n}_+) \cap C(\overline{\mathbb{S}^{n}_+}) $ such that
\begin{equation}\label{U0-f037}
U^0(X)= r^{-\frac{2\sg +\al}{p-1}} \varphi^0(\theta), 
\end{equation} 
where $X=(x ,t) =r \theta$ with $r=|X|$ and $\theta=\frac{X}{|X|}$.  A calculation  similar to the proof of Proposition \ref{P302} shows that  $\varphi^0$ satisfies 
\begin{equation}\label{Hom01-01}
\begin{cases}
- \theta_1^{2\sg-1} \text{div}_\theta (\theta_1^{1-2\sg} \nabla_\theta \varphi^0) + J_2 \varphi^0 =0~~~~~~~~~& \textmd{on} ~\mathbb{S}^n_+ ,\\
-\lim_{\theta_1\rightarrow 0^+} \theta_1^{1-2\sg} \partial_{\theta_1} \varphi^0 = \kappa_\sg (\varphi^0)^p~~~~~~~~& \textmd{on} ~ \partial \mathbb{S}^n_+,\\
\end{cases}
\end{equation}
where $\theta_1=\frac{t}{|X|}$ denotes the component of $\theta$ in the $t$ direction and  
$$
J_2=\frac{2\sg +\al}{p-1}\left(n-2\sg-\frac{2\sg + \al}{p-1}\right).
$$
By \eqref{TA02-02},    $\varphi^0$ also satisfies 
$$
0 < C_1 \leq \varphi^0(\theta) \leq C_2~~~~~~~~~~~ \textmd{on} ~~ \overline{\mathbb{S}^n_+}. 
$$
On the other hand, since $p < p_S(\al)$, from  Theorem 1.1 in \cite{LB2} we know that $U^0(x, t)$ is cylindrically  symmetric about  the origin.  In particular, the trace $u^0(x):=U^0(x, 0)$ is radially symmetric about the  origin.  Hence,  $\varphi^0$  is a positive constant on $\partial \mathbb{S}^n_+$.     
By Proposition \ref{4Le02}  we have 
$$
\varphi^0   \equiv  \omega_\alpha ~~~~~~~ \textmd{on} ~\overline{\mathbb{S}^n_+},  
$$
where $\omega_\alpha$ is defined as in Proposition  \ref{4Le02}.   It follows from the form \eqref{U0-f037} of $U^0$ and the proof of Proposition  \ref{4Le02} that 
$$
U^0(x, t) =C_{p,\sg,\al} \int_{\mathbb{R}^n} P_\sigma(x-y, t)  |y|^{-\frac{2\sg +\al}{p-1}} dy ~~~~~~ \textmd{for} ~(x, t) \in \mathbb{R}^{n+1}_+, 
$$
where $C_{p,\sg,\al}$ is given by \eqref{A}. 
Since the limiting function $U^0(x, t)$ is unique, we conclude that $U^\lambda(x, t)\rightarrow U^0(x, t)$ for any sequence $\lambda\rightarrow 0$ in $C^{\gamma/2}_{loc}(\overline{\mathbb{R}^{n+1}_+} \backslash \{0\})$.  In  particular,  
$$
|\lambda X|^{\frac{2\sg +\al}{p-1}} U(\lambda X)= |X|^{\frac{2\sg +\al }{p-1}} U^\lambda(X) \rightarrow U^0(X)~~~~~ \textmd{as}~~ \lambda \rightarrow 0
$$
uniformly for  $X \in \overline{\mathbb{S}^{n}_+}$.  This  immediately implies  that \eqref{Beh-0D-2020}  holds.

\vskip0.10in

{\bf Case 2:  $p_S(\al)  < p \leq \frac{n+2\sg+\al}{n-2\sg}$}.  
We consider the Kelvin transform 
$$
\widetilde{U} (Y)  =  \left( \frac{1}{|Y|} \right)^{n-2\sg} U\left( \frac{Y}{|Y|^2} \right)
$$
for $|Y|>1$. 
Then $\widetilde{U}$ satisfies 
\begin{equation}\label{Kel401}
\begin{cases}
-\textmd{div}(t^{1-2\sg}  \nabla \widetilde{U})=0~~~~~~~~~& \textmd{in} ~ \mathbb{R}_{+}^{n+1} \backslash \overline{\mathcal{B}_1^+},\\
\frac{\partial \widetilde{U}}{\partial \nu^\sg}(y, 0)=\kappa_\sg |y|^\vartheta  \widetilde{U}^p(y, 0)~~~~~~~~& \textmd{on} ~ B_1^c, 
\end{cases}
\end{equation}
where $\vartheta :=p(n-2\sg) - (n+2\sg+\al)$ and $B_1^c =\{x \in \mathbb{R}^n : |x| >1 \}$.  Using Theorem \ref{C-T2-2020}  we have 
\begin{equation}\label{Ubar01}
\frac{C_1}{|Y|^{(2\sg +\vartheta)/(p-1)}} \leq \widetilde{U} (Y) \leq \frac{C_2}{|Y|^{(2\sg +\vartheta)/(p-1)}} ~~~~~~~ \textmd{for} ~ |Y| ~\textmd{large}. 
\end{equation}
Note that 
$$
-2 \sg < \vartheta \leq  0
$$
due to $\frac{n+\al}{n -2\sg} < p$ and $p \leq \frac{n+2\sg +\al}{n-2\sg}$. 
Moreover, 
$$
\aligned
p > \frac{n + \vartheta}{n -2\sg} ~~~~~ & \Leftrightarrow ~~~~~ \al > -2\sg, \\
p< \frac{n+2\sg +2\vartheta}{n-2\sg}  ~~~~~ & \Leftrightarrow ~~~~~ p> \frac{n+2\sg+2\al}{n-2\sg}. 
\endaligned
$$
Therefore, after performing the Kelvin transform,   the new exponent $\vartheta$ satisfies  
$$
-2 \sg < \vartheta \leq 0 ~~~~~~~~ \textmd{and} ~~~~~~~~\frac{n+\vartheta}{n-2\sg} < p < \frac{n+2\sg +2\vartheta}{n-2\sg}. 
$$
For any $\lambda>0$, define 
$$
\widetilde{U}^\lambda(Y)=\lambda^{\frac{2\sg +\vartheta}{p-1}} \widetilde{U}(\lambda Y). 
$$
Then $\widetilde{U}^\lambda$ satisfies
$$
\begin{cases}
-\textmd{div}(t^{1-2\sg}  \nabla \widetilde{U}^\lambda)=0~~~~~~~~~& \textmd{in} ~ \mathbb{R}_{+}^{n+1} \backslash \overline{\mathcal{B}_{1/\lambda}^+},\\
\frac{\partial \widetilde{U}^\lambda}{\partial \nu^\sg}(y, 0)= \kappa_\sg |y|^{\vartheta} (\widetilde{U}^\lambda(y, 0))^p~~~~~~~~& \textmd{on} ~ 
\mathbb{R}^n \backslash  B_{1/\lambda}. 
\end{cases}
$$
By \eqref{Ubar01}, 
\begin{equation}\label{Ubar02}
C_1 |Y|^{-\frac{2\sg +\vartheta}{p-1}}\leq \widetilde{U}^\lambda (X) \leq C_2 |Y|^{-\frac{2\sg +\vartheta}{p-1}}~~~~~\textmd{in} ~ \mathbb{R}_{+}^{n+1} \backslash \overline{\mathcal{B}_{2/\lambda}^+}. 
\end{equation}
It follows from  Corollary 2.10 and Theorem 2.15  in \cite{J-L-X} that  there exists $\gamma >0$ such that for every $R > 1 > r > 0$, 
$$
\|\widetilde{U}^\lambda \|_{W^{1,2}(t^{1-2\sg}, \mathcal{B}_R^+ \backslash \overline{\mathcal{B}}_r^+)} + \|\widetilde{U}^\lambda\|_{C^\gamma (\mathcal{B}_R^+ \backslash \overline{\mathcal{B}}_r^+)} + \|\widetilde{u}^\lambda\|_{C^{2,\gamma}(B_R \backslash B_r)} \leq C(R, r),
$$
where $\widetilde{u}^\lambda(y) := \widetilde{U}^\lambda(y, 0)$ and $C(R, r)$ is independent of $\lambda$.  Then there is a subsequence $\lambda_k$ of $\lambda \rightarrow +\infty$  such that $\{U^{\lambda_k}\}$ converges to a nonnegative function $\widetilde{U}^\infty  \in W_{loc}^{1,2}(t^{1-2\sg}, \overline{\mathbb{R}^{n+1}_+} \backslash \{0\}) \cap C^\gamma_{loc}(\overline{\mathbb{R}^{n+1}_+} \backslash \{0\})$ satisfying 
$$
\begin{cases}
-\textmd{div}(t^{1-2\sg}  \nabla \widetilde{U}^\infty)=0~~~~~~~~~& \textmd{in} ~\mathbb{R}^{n+1}_+,\\
\frac{\partial \widetilde{U}^\infty}{\partial \nu^\sg}(y, 0)= \kappa_\sg |y|^\vartheta (\widetilde{U}^\infty(y ,0))^p~~~~~~~~& \textmd{on} ~\mathbb{R}^n \backslash \{0\}. 
\end{cases}
$$
By \eqref{Ubar02}   we have 
\begin{equation}\label{TA02-02-02}
C_1 |Y|^{-\frac{2\sg + \vartheta}{p-1}} \leq \widetilde{U}^\infty(Y) \leq C_2 |Y|^{-\frac{2\sg +\vartheta}{p-1}}~~~~~~\textmd{in} ~\overline{\mathbb{R}^{n+1}_+} \backslash \{0\}. 
\end{equation}
Moreover, by the scaling invariance of  $E(r; \widetilde{U})$ and Proposition \ref{4Le01}, we have for any $r>0$ that 
$$
E(r; \widetilde{U}^\infty)= \lim_{k\rightarrow \infty}E(r; \widetilde{U}^{\lambda_k})  = \lim_{k\rightarrow \infty}E(r\lambda_k; \widetilde{U})=\lim_{r\to +\infty}E(r; \widetilde{U}).
$$
That is, $E(r; \widetilde{U}^\infty)$ is a constant. It follows from  Proposition \ref{P302} that $\widetilde{U}^\infty$ is homogeneous of degree $-\frac{2\sg +\vartheta}{p-1}$.  Notice that we have  $p < p_S(\vartheta)$,  the same argument as in Case 1 gives  that $\widetilde{U}^\infty$ has the form 
$$
\widetilde{U}^\infty(y, t)=C_{p,\sg,\vartheta} \int_{\mathbb{R}^n} P_\sigma(y-z, t)  |z|^{-\frac{2\sg +\vartheta}{p-1}} dz ~~~~~~ \textmd{for} ~(y, t) \in \mathbb{R}^{n+1}_+, 
$$
where $C_{p,\sg,\vartheta}$ is given by \eqref{A}.  By the uniqueness of the  limiting  function  $\widetilde{U}^\infty$,   we conclude that $\widetilde{U}^\lambda(y, t)\rightarrow \widetilde{U}^\infty(y, t)$ for any sequence $\lambda\rightarrow +\infty$ in $C^\gamma_{loc}(\overline{\mathbb{R}^{n+1}_+} \backslash \{0\})$.  In particular, 
$$
|\lambda Y|^{\frac{2\sg +\vartheta}{p-1}} \widetilde{U}(\lambda Y)= |Y|^{\frac{2\sg +\vartheta}{p-1}} \widetilde{U}^\lambda(Y) \rightarrow \widetilde{U}^\infty(Y)~~~~~ \textmd{as}~~ \lambda \rightarrow +\infty
$$
uniformly for  $X \in \overline{\mathbb{S}^{n}_+}$.   Hence   we have 
$$
\widetilde{U}(Y) = \widetilde{U}^\infty(Y) \left( 1 + o(1) \right) ~~~~~  \textmd{as} ~~  |Y| \to +\infty.
$$
From  \eqref{AAA} we have that  $\Lambda(\tau)=\Lambda(-\tau)$.  Therefore 
$$
C_{p,\sg,\vartheta}=\left\{\Lambda\left(\frac{n-2\sg}{2} - \frac{2\sg + \vartheta}{p-1}\right)\right\}^{\frac{1}{p-1}} =\left\{\Lambda\left(\frac{2\sg + \al}{p-1} - \frac{n-2\sg}{2} \right)\right\}^{\frac{1}{p-1}}  = C_{p,\sg,\alpha}. 
$$
By the definition of  Kelvin transform,    we now easily get that  \eqref{Beh-0D-2020}  holds.  
This completes the proof of Theorem \ref{AB-T1-2020}.  
\hfill$\square$

\vskip0.10in

\noindent{\bf Proof of Theorem \ref{AB-T1}.} It follows from the extension theorem of Caffarelli-Silvestre \cite{C-S} and Theorem \ref{AB-T1-2020}.
\hfill$\square$

\vskip0.10in

Now  we give the proof of the uniqueness of global singular solutions in Theorem \ref{unique-2020}, which is similar to that of Theorem \ref{AB-T1-2020}.  But it is very different from the proof of the uniqueness theorem of Gidas-Spruck \cite{G-S} (See  Theorem 1.4 in \cite{G-S}).   

\vskip0.10in
\noindent{\bf Proof of Theorem \ref{unique-2020}.}
Suppose that $U$ is  a nonnegative weak solution of \eqref{Un-2020-01} and  the two singularities $0$ and $\infty$ of $u(x)$ are non-removable.   By Theorem \ref{C-T2-2020} there exist two positive constants $C_1$ and $C_2$ such that 
\be\label{Uniq-2020-001S}
\frac{C_1}{|X|^{(2\sg + \al)/(p-1)}} \leq U(X) \leq \frac{C_2}{|X|^{(2\sg + \al)/(p-1)}} ~~~~~ \textmd{near}  ~ X=0. 
\ee
Since the singularity at $\infty$ is non-removable, by using the  Kelvin transformation and Theorem \ref{C-T2-2020}, it is not difficult to prove that \eqref{Uniq-2020-001S} also holds near $X= \infty$. Hence we have 
\be\label{Uniq-2020-002S}
\frac{C_1}{|X|^{(2\sg + \al)/(p-1)}} \leq U(X) \leq \frac{C_2}{|X|^{(2\sg + \al)/(p-1)}} ~~~~~ \textmd{for} ~ \textmd{all}  ~ X \in \overline{\mathbb{R}^{n+1}_+}  \backslash \{0\}. 
\ee
Now we consider separately the subcritical case $p < p_S(\al)$ and the supercritical case $p>p_S(\al)$.   

\vskip0.10in

{\bf Case 1:  $p < p_S(\al)$}.  For any $\lambda >0$,  we define the scaling 
$$
U^\lambda(X)=\lambda^{\frac{2\sg +\al}{p-1}} U(\lambda X). 
$$
Then $U^\lambda$ also satisfies \eqref{Un-2020-01} and \eqref{Uniq-2020-002S}. The same arguments as in the proof of Theorem \ref{AB-T1-2020} imply that  there is a subsequence $\lambda_k$ of $\lambda \rightarrow 0$  such that $\{U^{\lambda_k}\}$ converges to a nonnegative function $U^0  \in W_{loc}^{1,2}(t^{1-2\sg}, \overline{\mathbb{R}^{n+1}_+} \backslash \{0\}) \cap C^\gamma_{loc}(\overline{\mathbb{R}^{n+1}_+} \backslash \{0\})$ satisfying \eqref{Un-2020-01} and \eqref{Uniq-2020-002S}.   By the scaling invariance of  $E(r; U)$ and Proposition \ref{4Le01}, we have for any $r>0$ that 
\begin{equation}\label{2020-Sca-011}
E(0^+; U) = \lim_{k\rightarrow \infty}E(r\lambda_k; U) = \lim_{k\rightarrow \infty}  E(r; U^{\lambda_k})  =  E(r; U^0).
\end{equation} 
Similar to the proof of  Theorem \ref{AB-T1-2020}, it follows from  Proposition \ref{P302}  and Proposition \ref{4Le02}  that $U^0$ has  the  explicit  expression 
\begin{equation}\label{20-Ex-01}
U^0(x, t) =C_{p,\sg,\al} \int_{\mathbb{R}^n} P_\sigma(x-y, t)  |y|^{-\frac{2\sg +\al}{p-1}} dy ~~~~~~ \textmd{for} ~(x, t) \in \mathbb{R}^{n+1}_+, 
\end{equation}
where $C_{p,\sg,\al}$ is given by \eqref{A}.  On the other hand, let $\lambda \to +\infty$, there is another subsequence $\lambda_k$ of $\lambda \rightarrow +\infty$  such that $\{U^{\lambda_k}\}$ converges to a nonnegative function $U^\infty  \in W_{loc}^{1,2}(t^{1-2\sg}, \overline{\mathbb{R}^{n+1}_+} \backslash \{0\}) \cap C^\gamma_{loc}(\overline{\mathbb{R}^{n+1}_+} \backslash \{0\})$ satisfying \eqref{Un-2020-01} and \eqref{Uniq-2020-002S}.   By the scaling invariance of  $E(r; U)$ and Proposition \ref{4Le01}, we have for any $r>0$ that 
\begin{equation}\label{2020-Sca-022}
E(\infty; U) = \lim_{k\rightarrow \infty}E(r\lambda_k; U) = \lim_{k\rightarrow \infty}  E(r; U^{\lambda_k})  =  E(r; U^\infty).
\end{equation} 
Similar to the above proof,  we obtain that $U^\infty$ also has  the  explicit  expression \eqref{20-Ex-01}.  
By \eqref{2020-Sca-011} and \eqref{2020-Sca-022} we have
$$
\lim_{r \to 0^+} E(r; U) = \lim_{r \to +\infty} E(r; U) \equiv E(r; U^0). 
$$
This together with the monotonicity of $E(r; U)$ on  $(0, \infty)$ implies  that $E(r; U)$ is a constant.  By a very similar argument as in  the proof of  Theorem \ref{AB-T1-2020},  it follows from Proposition \ref{P302}  and Proposition \ref{4Le02}  that $U(x, t)$ has  the  form \eqref{Un-2020-02}. 

\vskip0.10in

{\bf Case 2:  $p_S(\al)  < p \leq \frac{n+2\sg+\al}{n-2\sg}$}.  
We consider the Kelvin transform 
$$
\widetilde{U} (Y)  =  \left( \frac{1}{|Y|} \right)^{n-2\sg} U\left( \frac{Y}{|Y|^2} \right)
$$
for $Y \in \mathbb{R}^{n+1}_+ \backslash \{0\}$.  Denote $\vartheta :=p(n-2\sg) - (n+2\sg+\al)$.   Then the same arguments as in Case 1 give that $\widetilde{U}$ has the form 
$$
\widetilde{U}(y, t)=C_{p,\sg,\vartheta} \int_{\mathbb{R}^n} P_\sigma(y-z, t)  |z|^{-\frac{2\sg + \vartheta}{p-1}} dz ~~~~~~ \textmd{for} ~(y, t) \in \mathbb{R}^{n+1}_+, 
$$
where $C_{p,\sg,\vartheta}$ is given by \eqref{A}.  A simple calculation shows  that $U(x, t)$ satisfies  \eqref{Un-2020-02}.  The proof of Theorem \ref{unique-2020} is completed.  
\hfill$\square$

\section{Isolated Singularities at Infinity} 
In this section,  we prove Theorems  \ref{C-T3} and \ref{C-T4}.    

\vskip0.10in

\noindent{\bf Proof of Theorem \ref{C-T3}.}
Let $u(x)$ be a nonnegative solution of \eqref{Iny} with $\al >-2\sg$ and $1 < p <\frac{n+2\sg}{n-2\sg}$. We consider  the Kelvin transform
\be\label{K4F0}
\widetilde{u}(y)=\frac{1}{|y|^{n-2\sg}}u(\frac{y}{|y|^2})~~~~~~~~~~~\textmd{for} ~ y \in \mathbb{R}^n \backslash \{0\}.  
\ee
Then $\widetilde{u} \in C^2(B_1 \backslash \{0\}) \cap \mathcal{L}_\sigma(\mathbb{R}^n)$ and $\widetilde{u}$ satisfies 
$$
(-\Delta)^\sg \widetilde{u} =|y|^\varrho \widetilde{u}^p~~~~~~~~~~ \textmd{in} ~ B_1 \backslash \{0\}, 
$$
where $\varrho:=p(n-2\sg) -(n+2\sg +\al)$. 

\vskip0.10in

(1) If $1 < p < \frac{n+\al}{n-2\sg}$, then $\varrho < -2\sg$. By Corollary \ref{2Co1}  we have  $\widetilde{u}(y) \equiv 0$ in $B_1 \backslash \{0\}$,  this  implies that $u(x)\equiv 0$ for $|x| >1$. 

\vskip0.10in 

(2) If $\frac{n+\al}{n-2\sg} < p < \frac{n+2\sg}{n-2\sg}$, then $-2\sg < \varrho < -\al < 2\sg$ and 
$$
\frac{n+\varrho}{n-2\sg} = p -\frac{2\sg +\al}{n-2\sg} < p.
$$
It follows from Theorem \ref{C-T2} that either the singularity at $y=0$ is removable, or there exist $c_1, c_2 >0$ such that 
\be\label{N-R01}
\frac{c_1}{|y|^{(2\sg +\varrho)/(p-1)}} \leq \widetilde{u}(y) \leq  \frac{c_2}{|y|^{(2\sg +\varrho)/(p-1)}}~~~~~~~ \textmd{near} ~ y=0. 
\ee
If the singularity at $y=0$ is removable, then 
$$
u(x) =O\left( \frac{1}{|x|^{n-2\sg}} \right). 
$$
If \eqref{N-R01} holds, then
$$
\frac{c_1}{|x|^{(2\sg +\al)/(p-1)}} \leq u(x) \leq  \frac{c_2}{|x|^{(2\sg +\al)/(p-1)}}~~~~~~~ \textmd{near} ~ x=\infty. 
$$
This completes the proof. 
\hfill$\square$

\vskip0.10in

\noindent{\bf Proof of Theorem \ref{C-T4}.}
Let $u(x)$ be a positive  solution of \eqref{Iny} with $-2\sg < \al \leq 0$, $\frac{n+\al}{n-2\sg} < p  \leq \frac{n+2\sg +\al}{n-2\sg}$ and $p\neq\frac{n+2\sg +2\al}{n-2\sg}$. We define the Kelvin transform $\widetilde{u}(y)$ of $u(x)$ as in \eqref{K4F0}.  Then $\widetilde{u}(y)$ satisfies 
\be\label{KT0098}
(-\Delta)^\sg \widetilde{u} =|y|^\varrho \widetilde{u}^p~~~~~~~~~~ \textmd{in} ~ B_1 \backslash \{0\}, 
\ee
where $\varrho:=p(n-2\sg) -(n+2\sg +\al)$. Note that 
$$
\aligned
-2\sg < \varrho ~~~~~~ & \Leftrightarrow ~~~~~~ \frac{n+\al}{n-2\sg} < p, \\
\varrho \leq 0 ~~~~~~ & \Leftrightarrow ~~~~~~ p \leq \frac{n+2\sg +\al}{n-2\sg},\\
\frac{n+\varrho}{n-2\sg} < p ~~~~~~ & \Leftrightarrow ~~~~~~ -2\sg < \al,\\
p \leq  \frac{n+2\sg + \varrho}{n-2\sg} ~~~~~~ & \Leftrightarrow ~~~~~~ \al \leq 0,\\
p\neq \frac{n+2\sg + 2\varrho}{n-2\sg}  ~~~~~~ & \Leftrightarrow ~~~~~~ p\neq \frac{n+2\sg + 2\al}{n-2\sg}. 
\endaligned
$$
Hence, under the assumptions of Theorem \ref{C-T4}, we have 
$$
-2\sg < \varrho \leq 0, ~~~~~~ \frac{n+\varrho}{n-2\sg} < p \leq \frac{n+2\sg +\varrho}{n-2\sg}~~~~~~ \textmd{and} ~~~~~~p\neq\frac{n+2\sg +2\varrho}{n-2\sg}.
$$
Thus,  Theorem \ref{AB-T1} could  be applied to the equation \eqref{KT0098}  and we obtain that either the singularity near $y=0$ is removable, or  
\be\label{T4Lim0}
\lim_{|y|\to 0} |y|^{\frac{2\sg +\varrho}{p-1}}\widetilde{u}(y)= C_{p, \sg,\varrho}. 
\ee
If the singularity near $y=0$  is removable, then $\widetilde{u}(y)$  can be extended to a continuous function near the origin 0. Hence, there exists $\beta >0$ such that
$$
\lim_{|x| \to \infty} |x|^{n-2\sg}u(x) =\beta.
$$
If \eqref{T4Lim0} holds, then
$$
\lim_{|x| \to \infty} |x|^{\frac{2\sg +\al}{p-1}} u(x)=C_{p,\sg,\al}. 
$$
This completes the proof. 
\hfill$\square$

\end{document}